\immediate\write18{bibtex8 "\jobname.aux"}
\documentclass{article}
\usepackage{amssymb,amsmath,amsthm,enumerate,graphicx,mathtools,subcaption,tikz}
\usetikzlibrary{decorations.text,hobby}

\newtheorem*{theorem*}{Theorem}
\newtheorem{theorem}{Theorem}[section]
\newtheorem{lemma}[theorem]{Lemma}

\newtheorem{corollary}[theorem]{Corollary}
\newtheorem{remark}[theorem]{Remark}
\newtheorem{definition}[theorem]{Definition}

\def\Diff{\mathrm{Diff}}
\def\L{\mathcal{L}}
\def\U{\mathcal{U}}
\def\V{\mathcal{V}}
\def\W{\mathcal{W}}

\mathtoolsset{showonlyrefs}

\author{
  Andrew Clarke\\
  \texttt{andrew.michael.clarke@upc.edu}
  \and
  Jacques Fejoz\\
  \texttt{jacques.fejoz@dauphine.fr}
  \and
  Marcel Gu\`ardia\\
  \texttt{marcel.guardia@upc.edu}
}

\title{Topological Shadowing Methods in Arnold Diffusion: Weak Torsion and Multiple Time Scales}

\begin{document}
\maketitle
\begin{abstract}
  Consider a symplectic map which possesses a normally hyperbolic invariant manifold of any even dimension with transverse homoclinic channels. We develop a topological shadowing argument to prove the existence of Arnold diffusion along the invariant manifold, shadowing some iterations of the inner dynamics carried by the invariant manifold and the outer dynamics induced by the stable and unstable foliations. In doing so, we generalise an idea of Gidea and de la Llave in \cite{gidea2006topological}, based on the method of correctly aligned windows and a so-called transversality-torsion argument. Our proof permits that the dynamics on the invariant manifold satisfy only a non-uniform twist condition, and, most importantly for applications, that the splitting of separatrices be small in certain directions and thus the associated drift in actions very slow; diffusion occurs in the directions of the manifold having non-small splitting. Furthermore we provide estimates for the diffusion time.
\end{abstract}

\section{Introduction}

Arnold diffusion, as first exposed in~\cite{Arnold:1964} has become a major subject of study for nearly integrable Hamiltonian systems. This mechanism epitomises how an integrable (``stable'') system can become unstable, with actions varying slowly (as permitted by Nekhoroshev's theorem) but substantially, as opposed to what would happen for an integrable perturbation~\cite{Bolotin:1999, Delshams:2000}. A key idea has been to focus on a normally hyperbolic invariant cylinder. The cylinder is indeed a more robust object than individual hyperbolic invariant tori that may lie inside the cylinder and which Arnold initially used. Another key addition to Arnold's mechanism is the random iteration of the dynamics carried by the cylinder and the ``outer'' dynamics obtained at the limit by following unstable and stable leaves of the cylinder (see figure~\ref{figure_homoclinicchannel}). Both these ideas appear in Moeckel's work~\cite{Moeckel:2002}.
These ideas were formalised by  Delshams, de la Llave and Seara with the invention of the scattering map which encodes the outer dynamics. This allowed them to solve the large gap problem, when two (primary) invariant tori are too far away for their unstable and stable manifolds to meet \cite{Delshams:2006a}. In these last decades, this approach has been successfully implemented by a number of authors, using either a geometric description~\cite{Delshams:2006a,Delshams:16,Gelfreich:2008,GelfreichTuraev2017,Treschev:2004, Treschev:2012} or a variational one~\cite{Cheng:2004,Cheng:2009}.

Another step consists in dealing with the a priori stable case, where the normally hyperbolic cylinder appears with the perturbation itself~\cite{Bernard08,Kaloshin:2016,Cheng:2017,Cheng:2009,Kaloshin:2020}. In this article we deal with the a priori unstable case only. 

One of the main difficulties lies in proving the existence of orbits shadowing random iterations of the inner and outer maps. Shadowing results can be achieved using different tools: variational methods (see, for instance, \cite{Berti2,Berti1,Bessi96}), modern versions of the Lambda lemma  \cite{FontichM00, gidea2020general,Sabbagh15}, or with topological techniques such as the  the correctly aligned windows method.

The advantge of the topological methods is that they require relatively little information regarding the dynamics on the normally hyperbolic invariant cylinder. In particular, no knowlege of invariant quasiperiodic tori is expected. As far as the authors know, correctly aligned windows were imagined by Conley and Easton \cite{Conley:71}, and the first application of the correctly aligned windows method to Arnold diffusion problems is the paper \cite{gidea2006topological}. In this paper, Gidea and de la Llave use the method of correctly aligned windows to prove the existence of diffusion orbits in a priori unstable Hamiltonian systems and to construct orbits with an unbounded growth of energy for the Mather problem (that is a geodesic flow with a time dependent potential). This second model is usually said to be \emph{a priori chaotic}. In both settings the normally hyperbolic invariant cylinder is two-dimensional. Moreover, they assume that the induced dynamics on the cylinder satisfies a twist property, and the twist is uniform with respect to the perturbative parameter.

The purpose of this paper is to generalise \cite{gidea2006topological} from several points of view:
\begin{itemize}
\item The normally hyperbolic cylinder may be of any (even) dimension.
\end{itemize}
More importantly:
\begin{itemize}
\item The twist property satisfied by the inner dynamics may be weak, i.e. the twist may vanish when the perturbative parameter goes to 0.
\item We may split the actions between two groups, faster actions and slower actions, and ignore the latter. Indeed, proving ``partial transversality'' of the  invariant manifolds of the normally hyperbolic cylinder along the fast ones is enough to achieve drift in these directions whereas the slow directions can be treated  as a black box, with no control of their instability rate. 
\end{itemize}

These improvements are crucial since such behaviour is exhibited naturally in
physical models, particularly in models with multiple time scales as happens often, for instance, 
in Celestial Mechanics. 

Many of the known shadowing mechanisms rely on rather strong assumptions, both on the inner dynamics in the cylinder and on the transversality of the associated stable and unstable invariant foliations. Often such hypotheses are difficult to verify when  multiple time scales are present. The results presented in this paper are quite flexible and can be applied to rather general multiple time scale settings.

In particular, consider an analytic nearly integrable Hamiltonian system with multiple time scales such that some of the angles perform fast non-resonant rotation and therefore its conjugate actions are very slow whereas some other angles are not fast. It is well known that, in the fast directions, the transversality between the invariant manifolds is exponentially small and therefore very difficult to analyze. At the same time, thanks to averaging theory, one can make the dynamics in the conjugate actions much slower. The results in this paper allow us to obtain diffusing orbits along the actions conjugated to the slow angles even if one does not have ``full transversality'' of the invariant manifolds, i.e. no transverality in the fast directions (see Theorem \ref{theorem_main2} below). 

All the improvements achieved in the present papers are needed to construct diffusing orbits in the 4 Body Problem along secular resonances. In the companion paper \cite{Clarke:22}, the authors prove the
existence of orbits of the 4 body problem in both the planetary regime (one massive body and three bodies with small mass) and hierarchical regime (bodies increasingly separated) such that some of the bodies have osculating eccentricites and some
mutual inclinations drifting with no constraint. This is the first analytical result
of unstable motions in an $N$ body problem in the planetary regime. 

Secular resonances are those given by the secular angles, that is the angles which are constant for the two body problem: the argument of the perihelion and the longitude of the ascending node of each of the bodies. Such angles are much slower than the mean anomalies of the bodies which perform fast rotation. Therefore, we are exactly in the multiple time-scale setting described above. In \cite{Clarke:22}, we are able to construct diffusing orbits in the actions conjugated to the secular angles, i.e. the osculating eccentricities and mutual inclinations of the bodies, without having to control the dynamics on semimajor axis directions.

\subsection*{Acknowledgments}
A. Clarke and M. Guardia are supported by the European Research Council (ERC) under the European Union's Horizon 2020 research and innovation programme (grant agreement No. 757802). M. Guardia is also supported by the Catalan Institution for Research and Advanced Studies via an ICREA Academia Prize 2019. This work is also supported by the Spanish State Research Agency, through the Severo Ochoa and María de Maeztu Program for Centers and Units of Excellence in R\&D (CEX2020-001084-M).

This work is also partially supported by the project of the French Agence Nationale pour la Recherche CoSyDy (ANR-CE40-0014).

\section{Set-up, Assumptions, and Results}

\subsection{Definitions}\label{sec_prelimdefns}

Let $M$ be a $C^r$ manifold of dimension $d$ where $r \geq 1$. Let $F \in \Diff^1(M)$, and let $\Lambda$ be a submanifold of $M$. 
\begin{definition}
We call $\Lambda$ a \emph{normally hyperbolic invariant manifold} for $F$ if $\Lambda$ is $F$-invariant, and there are
\begin{equation}\label{eq_hyperbolicityparameters1}
0 < \lambda_- < \lambda_+ < \lambda_0 < 1 < \mu_0 < \mu_- < \mu_+
\end{equation}
and an invariant splitting of the tangent bundle
\begin{equation}
T_{\Lambda} M = TM \oplus E^s \oplus E^u
\end{equation}
such that:
\begin{equation}\label{eq_normalhyperbolicitydef}
\begin{split}
\lambda_-^n \| v \| \leq \| D F^n (x) v \| \leq \lambda_+^n \| v \|  & \textrm{ for all } x \in \Lambda, v \in E^s_x, n \in \mathbb{N} \\
\mu_-^n \| v \| \leq \| D F^n (x) v \| \leq \mu_+^n \| v \| &  \textrm{ for all } x \in \Lambda, v \in E^u_x, n \in \mathbb{N} \\
\lambda_0^{|n|} \| v \| \leq \| D F^n (x) v \| \leq \mu_0^{|n|} \| v \| & \textrm{ for all } x \in \Lambda, v \in T_x \Lambda, n \in \mathbb{Z}.
\end{split}
\end{equation}
\end{definition}

The results presented in this paper also apply to the case where $\Lambda$ is a manifold with boundary, and indeed the case where $\Lambda$ is only \emph{locally} invariant: this means that there is a neighbourhood $V$ of $\Lambda$ such that orbits of points in $\Lambda$ stay in $\Lambda$ until they leave $V$. In the latter case orbits may escape through the boundary. 

This definition guarantees the existence of stable and unstable invariant manifolds $W^{s,u} (\Lambda)\subset M$ defined as follows. The local stable manifold $W^{s}_{\mathrm{loc}}(\Lambda)$ is the set of points in a small neighbourhood of $\Lambda$ whose forward iterates never leave the neighbourhood, and tend exponentially to $\Lambda$. The local unstable manifold $W^{u}_{\mathrm{loc}}(\Lambda)$ is the set of points in the neighbourhood whose backward iterates stay in the neighbourhood and tend exponentially to $\Lambda$.  We then define
\begin{equation}
W^s(\Lambda) = \bigcup_{i=0}^{\infty} F^{-i} \left( W^{s}_{\mathrm{loc}}(\Lambda) \right), \quad W^u(\Lambda) = \bigcup_{i=0}^{\infty} F^{i} \left( W^{u}_{\mathrm{loc}}(\Lambda) \right).
\end{equation}
On the stable and unstable manifolds we have the strong stable and strong unstable foliations, the leaves of which we denote by $W^{s,u}(x)$ for $x \in \Lambda$. For each $x \in \Lambda$, the leaf $W^s(x)$ of the strong stable foliation is tangent at $x$ to $E^s_x$, and the leaf $W^u(x)$ of the strong unstable foliation is tangent at $x$ to $E^u_x$. Moreover the foliations are invariant in the sense that $F^i \left( W^s (x) \right) = W^s \left( F^i (x) \right)$ and $F^i \left( W^u (x) \right) = W^u \left( F^i (x) \right)$ for each $x \in \Lambda$ and $i \in \mathbb{Z}$. We thus define the \emph{holonomy maps} $\pi^{s,u} : W^{s,u} (\Lambda) \to \Lambda$ to be projections along leaves of the strong stable and strong unstable foliations. That is to say, if $x \in W^s (\Lambda)$ then there is a unique $x_+ \in \Lambda$ such that $x \in W^s(x_+)$, and so $\pi^s(x)=x_+$. Similarly, if $x \in W^u (\Lambda)$ then there is a unique $x_- \in \Lambda$ such that $x \in W^u(x_-)$, in which case $\pi^u(x)=x_-$.

Now, suppose that $x \in \left(W^s(\Lambda) \pitchfork W^u (\Lambda)\right) \setminus \Lambda$ is a transverse homoclinic point such that $x \in W^s(x_+) \cap W^u(x_-)$. We say that the homoclinic intersection at $x$ is \emph{strongly transverse} if
\begin{equation}
\begin{split} \label{eq_strongtransversality}
T_x W^s (x_+) \oplus T_x \left( W^s(\Lambda) \cap W^u(\Lambda) \right) = T_x W^s (\Lambda), \\
T_x W^u (x_-) \oplus T_x \left( W^s(\Lambda) \cap W^u(\Lambda) \right) = T_x W^u (\Lambda).
\end{split}
\end{equation}
In this case we can take a sufficiently small neighbourhood $\Gamma$ of $x$ in $W^s(\Lambda) \cap W^u(\Lambda)$ so that \eqref{eq_strongtransversality} holds at each point of $\Gamma$, and the restrictions to $\Gamma$ of the holonomy maps are bijections onto their images. We call $\Gamma$ a \emph{homoclinic channel} (see Figure \ref{figure_homoclinicchannel}). We can then define the scattering map as follows \cite{delshams2008geometric}.

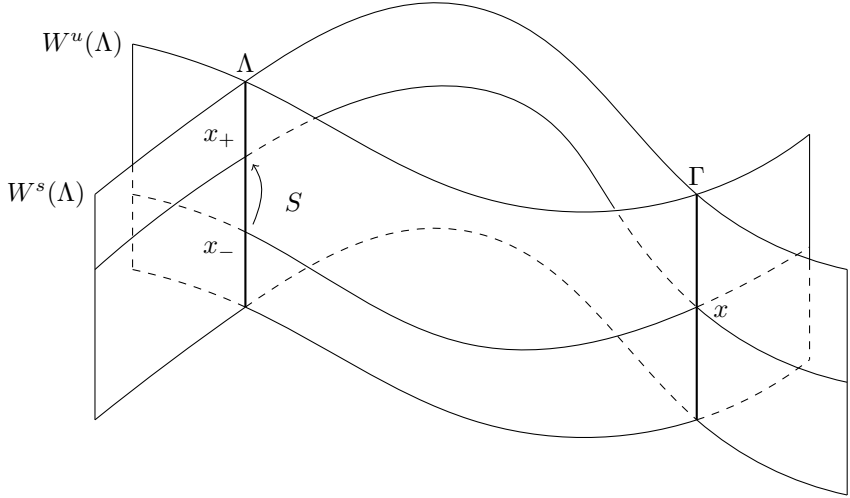
\begin{figure}
  \begin{tikzpicture}[use Hobby shortcut]

    \draw (0,0) .. (2,1.5) .. (5,2.5) .. (8,0) .. (10,-1);
    \draw (0,0) node[anchor=east] {$W^s(\Lambda)$};
    \draw (0,-3) .. (2,-1.5) .. ([blank=soft]5,-0.5) .. ([blank=soft]8,-3) .. (10,-4);
    \draw[dashed,use previous Hobby path={invert soft blanks,disjoint}]; 
    \draw (0.5,2) .. (2,1.5) .. (5,0) .. (8,0) .. (9.5,0.8);
    \draw (0.5,2) node[anchor=east] {$W^u(\Lambda)$};
    \draw (0.5,-1) .. ([blank=soft]2,-1.5) .. (5,-3) .. (8,-3)
    .. ([blank=soft]9.5,-2.2); 
    \draw[dashed,use previous Hobby path={invert soft blanks,disjoint}]; 
    \draw (0,0) -- (0,-3);
    \draw (10,-1) -- (10,-4);
    \draw (0.5,2) .. (0.5,0.4) .. ([blank=soft]0.5,-1);
    \draw[dashed,use previous Hobby path={invert soft blanks,disjoint}]; 
    \draw (9.5,0.8) .. (9.5,-0.9) .. ([blank=soft]9.5,-2.2);
    \draw[dashed,use previous Hobby path={invert soft blanks,disjoint}]; 
    \draw[thick] (2,-1.5) -- (2,1.5) node[anchor=south] {$\Lambda$};
    \draw[thick] (8,-3) -- (8,0) node[anchor=south] {$\Gamma$};
    \draw (0,-1) .. (2,0.5) .. ([blank=soft]2.9,1) .. (6,1)
    .. (6.85,-0.1) .. ([blank=soft]8,-1.5) .. (10,-2.5);
    \draw[dashed,use previous Hobby path={invert soft blanks,disjoint}]; 
    \draw (0.5,0) .. ([blank=soft]2,-0.5) .. (5,-2) .. (8,-1.5) .. ([blank=soft]9.5,-0.7);
    \draw[dashed,use previous Hobby path={invert soft blanks,disjoint}]; 
    \draw (2,-0.5) node[anchor=north east] {$x_-$};
    \draw (2,0.5) node[anchor=south east] {$x_+$};
    \draw[->] (2.1,-0.4) .. (2.4,-0.1) node [anchor=west] {$S$} .. (2.1,0.4);
    \draw (8.1,-1.55) node[anchor=west] {$x$};
  \end{tikzpicture}
  \caption{The scattering map $S$ takes a point $x_- \in \Lambda$, follows the unique leaf of the strong unstable foliation passing through $x_-$ to the point $x$ in the homoclinic channel $\Gamma$, and from there follows the unique leaf of the strong stable foliation passing through $x$ to the point $x_+$ on $\Lambda$.} \label{figure_homoclinicchannel}
\end{figure}

\begin{definition}
Let $y_-\in \pi^u \left( \Gamma \right)$, let $y = \left(\left. \pi^u \right|_{\Gamma} \right)^{-1} (y_-)$, and let $y_+ = \pi^s(y)$. The \emph{scattering map} $S : \pi^u (\Gamma) \to \pi^s (\Gamma)$ is defined by
\begin{equation}
S = \pi^s \circ \left( \pi^u \right)^{-1} : y_- \longmapsto y_+.
\end{equation}
\end{definition}

Suppose now that the smoothness $r$ of $M$ and $F$ is at least $2$, suppose the normally hyperbolic invariant manifold $\Lambda$ is a $C^r$ submanifold of $M$, and suppose instead of condition \eqref{eq_hyperbolicityparameters1} we have the stronger condition
\begin{equation}\label{eq_hyperbolicityparameters2}
0 < \lambda_- < \lambda_+ < \lambda_0^r < 1 < \mu_0^r < \mu_- < \mu_+
\end{equation}
on the hyperbolicity parameters. This \emph{large spectral gap condition} implies $C^{r-1}$ regularity of the strong stable and strong unstable foliations \cite{hirsch1970invariant}, which in turn implies that the scattering map $S$ is $C^{r-1}$ \cite{delshams2008geometric}.

\begin{remark}
In general, the scattering map is not globally defined. The homoclinic intersection of invariant manifolds can give rise to very complicated domains of definition of the scattering map, and the general case is that there are many branches $S_{\alpha}$ of the map defined on sets $U_j \subset \Lambda$. The sets $U_j$ may or may not overlap, and the scattering map may have singularities on $\partial U_j$. While the results of this paper apply to the general case, we simply write $S: U \to \Lambda$ to denote the scattering map to avoid awkward notation. 
\end{remark}

Suppose the map $F$ depends smoothly on a small parameter $\epsilon$. We point out that all objects defined thus far may vary with $\epsilon$, but, to simplify notation, we do not explicitly write $\epsilon$ as a subscript or argument. Indeed, as $\epsilon$ varies, the perturbed ($\epsilon>0$) normally hyperbolic invariant manifold can be written as a graph over the unperturbed ($\epsilon=0$) manifold as a result of Fenichel theory, and so we can continue to use the coordinates from the original unperturbed manifold \cite{fenichel1971persistence,fenichel1974asymptotic,fenichel1977asymptotic}. 

Suppose the scattering map $S$ is defined relative to a homoclinic channel $\Gamma$ for all sufficiently small $\epsilon >0$. We allow for the possibility that the angle between $W^{s,u}(\Lambda)$ along the homoclinic channel $\Gamma$ goes to 0 as $\epsilon \to 0$. Denote by $\alpha (v_1, v_2)$ the angle between two vectors $v_1, v_2$ in the direction that yields the smallest result (i.e. $\alpha (v_1,v_2) \in [0, \pi]$). For $x \in \Gamma$, let 
\begin{equation}
\alpha_{\Gamma} (x) = \inf \alpha (v_+, v_-)
\end{equation}
where the infimum is over all $v_+ \in T_x W^s (\Lambda)^{\perp}$ and $v_- \in T_x W^u (\Lambda)^{\perp}$ such that $\| v_{\pm} \| = 1$. 
\begin{definition}
For $\sigma \geq 0$, we say that \emph{the angle of the splitting along $\Gamma$ is of order $\epsilon^{\sigma}$} if there is a positive constant $C$ (independent of $\epsilon$) such that
\begin{equation}
\alpha_{\Gamma} (x) \geq C \epsilon^{\sigma}
\end{equation}
for all $x \in \Gamma$.
\end{definition}

Suppose now that the normally hyperbolic invariant manifold $\Lambda$ is diffeomorphic to $\mathbb{T}^n \times [0,1]^n$, and denote by $(q,p) \in \mathbb{T}^n \times [0,1]^n$ smooth coordinates on $\Lambda$. Suppose the maps $F$, and thus $f \coloneqq F|_{\Lambda}$, depend on the small parameter $\epsilon$.
\begin{definition}\label{def_nearlyintegrabletwist}
We say that $f: \Lambda \to \Lambda$ is a \emph{near-integrable twist map} if there is some $k \in \mathbb{N}$ such that
\begin{equation}\label{eq_inttwistmap}
f:
\begin{cases}
\bar{q} = q + g(p) + O(\epsilon^k) \\
\bar{p} = p + O(\epsilon^k)
\end{cases}
\end{equation}
where
\begin{equation}
\det D g (p) \neq 0
\end{equation}
for all $p \in [0,1]^n$, and where the higher order terms are uniformly bounded in the $C^1$ topology. If the higher order terms are 0 then $f$ is an \emph{integrable twist map}.
\end{definition}

\begin{remark}
The assumption that $\det Dg(p) \neq 0$ is sometimes referred to as a \emph{local} twist property; see for example Section 4, Chapter 23 of \cite{gole2001symplectic}. Note that we do not require convexity. 
\end{remark}

It follows from the definition that if $f : \Lambda \to \Lambda$ is a near-integrable twist map, then there exist twist parameters $T_+ > \widetilde{T}_- >0$ such that 
\begin{equation}\label{eq_twistcondition}
\widetilde{T}_- \| v \| \leq \left\| D g(p) v \right\| \leq T_+ \| v \|
\end{equation}
for all $p \in [0,1]^n$ and all $v \in \mathbb{R}^n$. We can always choose $T_+$ to be independent of $\epsilon$. Our formulation of the problem allows the parameter $\widetilde{T}_-$ to depend on $\epsilon$: there is $\tau \in \mathbb{N}_0$ and a strictly positive constant $T_-$ (independent of $\epsilon$) such that $\widetilde{T}_- = \epsilon^{\tau} T_-$. 
\begin{definition}
Suppose $f : \Lambda \to \Lambda$ is a near-integrable twist map. Denote by $T_+ > \widetilde{T}_- = \epsilon^{\tau} T_- >0$ the twist parameters. We say that $f$ satisfies:
\begin{itemize}
\item
A \emph{uniform twist condition} if $\tau=0$; 
\item
A \emph{non-uniform twist condition (of order $\epsilon^{\tau}$)} if $\tau>0$, and the order $\epsilon^k$ of the error terms in the definition of the near-integrable twist map $f$ is such that $k > \tau$.
\end{itemize}
\end{definition}

In the coordinates $(q,p)$, we may define a foliation of $\Lambda$, the leaves of which are given by
\begin{equation}\label{eq_foliationleaves}
\L (p^*) = \left\{ (q,p) \in \Lambda : p=p^* \right\}.
\end{equation}
If $f : \Lambda \to \Lambda$ is a near-integrable twist map in the sense of Definition \ref{def_nearlyintegrabletwist}, then each leaf of the foliation is almost invariant under $f$, up to terms of order $\epsilon^k$, where $k$ is as in Definition \ref{def_nearlyintegrabletwist}.

Suppose we have a scattering map $S$ defined on an open set $U$ in $\Lambda$, and suppose the large spectral gap condition \eqref{eq_hyperbolicityparameters2} holds, so $S$ is $C^1$.

\begin{definition}
We say that the scattering map $S$ is \emph{transverse to leaves along leaves (with respect to the leaves \eqref{eq_foliationleaves} of the foliation of $\Lambda$)} if for all $p_0^* \in [0,1]^n$ there is $c > 0$ (which may depend on $\epsilon$) and $p_1^* \in [0,1]^n$ such that for any $p^* \in [0,1]^n$ satisfying 
\begin{equation} \label{eq_pnbhdfortransversality}
\| p^* - p_1^* \| < c
\end{equation}
we have
\begin{equation}
S \left( \L (p_0^*) \cap U \right) \cap \L (p^*) \neq \emptyset
\end{equation}
and $S \left( \L (p_0^*) \cap U \right)$ is transverse to $\L (p^*)$ at some point where they intersect. In this case we say that \emph{the angle of transversality is of order $\epsilon^{\upsilon}$} if there are positive constants $C,c_*$ (independent of $\epsilon$) such that
\begin{equation}
\inf \alpha (v_0,v) \geq C \epsilon^{\upsilon}, \quad c=\epsilon^{\upsilon} c_*
\end{equation}
where the infimum is taken over all $v_0 \in T_x S \left( \L (p_0^*) \cap U \right)$ and $v \in T_x \L (p^*)$ such that $\| v_0 \| = \| v \| = 1$, for some $x \in S \left( \L (p_0^*) \cap U \right) \cap \L (p^*)$, as well as for each $p_0^* \in [0,1]^n$ and each $p^* \in [0,1]^n$ satisfying \eqref{eq_pnbhdfortransversality}.
\end{definition}

\subsection{Statement of Theorem \ref{theorem_main1}}

Let $M$ be a $C^r$ manifold of dimension $2(m+n)$ where $r \geq 4$ and $m,n \in \mathbb{N}$. Suppose $F \in \Diff^4 (M)$ has a normally hyperbolic invariant manifold $\Lambda \subset M$ of dimension $2n$ satisfying the large spectral gap condition \eqref{eq_hyperbolicityparameters2}. Suppose $\dim W^s (\Lambda) = \dim W^u(\Lambda)=m+2n$. Suppose $F$ depends on a small parameter $\epsilon$. We make the following further assumptions. 
\begin{enumerate}[{[}{A}1{]}]
\item
The stable and unstable manifolds $W^{s,u} (\Lambda)$ have a strongly transverse homoclinic intersection along a homoclinic channel $\Gamma$, and so we have an open set $U \subseteq \Lambda$ and a scattering map $S : U \to \Lambda$. The angle of the splitting along $\Gamma$ is of order $\epsilon^{\sigma}$. 
\item
The inner map $f = F|_{\Lambda}$ is a near-integrable twist map with error terms of order $\epsilon^k$ satisfying a non-uniform (or uniform) twist condition of order $\epsilon^{\tau}$. 
\item
The scattering map $S$ is transverse to leaves along leaves (with respect to the leaves \eqref{eq_foliationleaves} of the foliation of $\Lambda$), and the angle of transversality is of order $\epsilon^{\upsilon}$. 
\end{enumerate}

\begin{theorem} \label{theorem_main1} 
Fix $\eta >0$, let $\epsilon > 0$ be sufficiently small, and suppose
\begin{equation} \label{eq_kavgcondition}
k \geq 2 \left( \rho + \tau \right) + 1
\end{equation}
where
\begin{equation} \label{eq_rhodef}
\rho = \max \{2 \sigma,  2\upsilon, \tau \}.
\end{equation}
Choose $\{ p_j^* \}_{j=1}^{\infty} \subset [0,1]^n$ such that 
\begin{equation}
S \left( \L_j \cap U \right) \cap \L_{j+1} \neq \emptyset,
\end{equation}
and $S \left( \L_j \cap U \right)$ is transverse to $\L_{j+1}$, where $\L_j = \L (p_j^*)$. Suppose the distance between $\L_j$ and $\L_{j+1}$ is of order $\epsilon^{\upsilon}$ for each $j$. Then there are $\{z_i\}_{i=1}^{\infty} \subset M$ and $n_i \in \mathbb{N}$ such that 
\begin{equation}
z_{i+1} = F^{n_i} (z_i)
\end{equation}
and 
\begin{equation}
d ( z_i, \L_i) < \eta.
\end{equation}
Moreover, the time to move a distance of order 1 in the $p$-direction is bounded from below by a term of order
\begin{equation} \label{eq_timeestimate1}
\epsilon^{- \rho - \tau - \upsilon}.
\end{equation}
\end{theorem}

\begin{remark}
Note that we do not assume symplecticity of the map $F$. We do, however, assume some properties displayed by symplectic maps; for example, the even dimension of the phase space and of the stable and unstable manifolds.
\end{remark}

\subsection{Statement of Theorem \ref{theorem_main2}}

Let $M$ be a $C^r$ manifold of dimension $2(m+n)$ where $r \geq 4$ and $m,n \in \mathbb{N}$. Let $\Sigma = \mathbb{T}^{\ell_1} \times [0,1]^{\ell_2}$ for some $\ell_1, \ell_2 \in \mathbb{N}_0$, and denote by $(\theta, \xi) \in \mathbb{T}^{\ell_1} \times [0,1]^{\ell_2}$ coordinates on $\Sigma$. Write $\widetilde{M} = M \times \Sigma$. Suppose $\Psi \in \Diff^4 \left(\widetilde{M} \right)$ such that 
\begin{equation}
\Psi (z, \theta, \xi) = \left( G(z, \theta, \xi), \phi (z, \theta, \xi) \right)
\end{equation}
where $z \in M$, $G \in C^4 \left( \widetilde{M}, M \right)$, and $\phi \in C^4 \left( \widetilde{M}, \Sigma \right)$. Suppose $\Psi$ depends on a small parameter $\epsilon$. We make the following assumptions on $\Psi$. 
\begin{enumerate}[{[}{B}1{]}]
\item
There is some $L \in \mathbb{N}$ such that
\begin{equation}
G (z, \theta, \xi) = \widetilde{G}(z; \xi ) + O \left(\epsilon^L \right)
\end{equation}
where the higher order terms are uniformly bounded in the $C^4$ topology, and for each $\xi \in [0,1]^{\ell_2}$ the map
\begin{equation}
\widetilde{G} ( \cdot ; \xi ): z \in M \longmapsto \widetilde{G} ( z ; \xi ) \in M
\end{equation}
satisfies the assumptions [A1-3] of Theorem \ref{theorem_main1}.
\item
Moreover, the map $\phi$ has the form
\begin{equation}
\phi:
\begin{cases}
\bar{\theta} = \phi_1 (z, \theta, \xi) \\
\bar{\xi} = \phi_2(z, \theta, \xi) = \xi + O \left( \epsilon^L \right)
\end{cases}
\end{equation}
where the higher order terms are uniformly bounded in the $C^4$ topology. 
\end{enumerate}

Results from \cite{delshams2008geometric} imply that $\Psi$ has a normally hyperbolic invariant manifold $\widetilde{\Lambda}$ that is $O \left(\epsilon^L \right)$ close in the $C^4$ topology to $\Lambda \times \Sigma$ where $\Lambda \subset M$ is the normally hyperbolic invariant manifold of $\widetilde{G} ( \cdot ; \xi )$. Moreover there is an open set $\widetilde{U} \subset \widetilde{\Lambda}$ and a scattering map $\widetilde{S} : \widetilde{U} \to \widetilde{\Lambda}$ such that the $z$-component of $\widetilde{S}(z, \theta, \xi)$ is $O \left( \epsilon^L \right)$ close in the $C^3$ topology to $S \left( z ; \xi \right)$ where $S \left( \cdot ; \xi \right) : U \to \Lambda$ is the scattering map corresponding to $\widetilde{G} ( \cdot ; \xi )$. 

We use the coordinates $(q,p, \theta, \xi)$ on $\widetilde{\Lambda}$ where $(q,p)$ are the coordinates on $\Lambda$ and $(\theta, \xi)$ are the coordinates on $\Sigma$. Notice that the sets
\begin{equation}
\widetilde{\L} \left(p^*, \xi^* \right) = \left\{ (q,p, \theta, \xi) \in \widetilde{\Lambda}: p = p^*, \xi=\xi^* \right\} = \L \left(p^* \right) \times \mathbb{T}^{\ell_1} \times \left\{ \xi^* \right\}
\end{equation}
for $p^* \in [0,1]^n$ and $\xi^* \in [0,1]^{\ell_2}$ define the leaves of a foliation of $\widetilde{\Lambda}$, where $\L(p^*)$ are the leaves of the foliation of $\Lambda$ defined by \eqref{eq_foliationleaves}. 

\begin{theorem}\label{theorem_main2}
Fix $\eta>0$ and $K\in\mathbb{N}$ and let $\epsilon >0$ be sufficiently small. Choose $N \in \mathbb{N}$ satisfying
\[
 N\leq \frac{1}{\epsilon^K}
\]
and $p_1^*, \ldots, p_N^* \in [0,1]^n$ as in Theorem \ref{theorem_main1}, and choose $\xi^*_1 \in \mathrm{Int} \left([0,1]^{\ell_2} \right)$ such that
\begin{equation}
S \left( \L_j \cap U ; \xi^*_1 \right) \cap \L_{j+1} \neq \emptyset
\end{equation}
and $S \left( \L_j \cap U; \xi^*_1 \right)$ is transverse to $\L_{j+1}$, where $\L_j = \L (p_j^*)$. Suppose the distance between $\L_j$ and $\L_{j+1}$ is of order $\epsilon^{\upsilon}$ for each $j$, and $L >0$ is sufficiently large. Then there are $\xi^*_2, \ldots, \xi^*_N \in [0,1]^{\ell_2}$ such that, with $\widetilde{\L}_j = \widetilde{\L} \left(p^*_j, \xi^*_j \right)$, there are $w_1, \ldots, w_N \in \widetilde{M}$ and $n_i \in \mathbb{N}$ such that the $\xi$ component of $w_1$ is $\xi^*_1$,
\begin{equation}
w_{i+1} = \Psi^{n_i} (w_i),
\end{equation}
and
\begin{equation}
d \left(w_i, \widetilde{\L}_i \right) < \eta
\end{equation}
where $\rho, \sigma, \tau$ are as in the statement of Theorem \ref{theorem_main1}. Moreover, the time to move a distance of order 1 in the $p$-direction is of order
\begin{equation} \label{eq_timeestimate2}
\epsilon^{- \rho - \tau - \upsilon}.
\end{equation}
\end{theorem}

\begin{remark}
Note that the transition chain obtained in Theorem \ref{theorem_main2} is only of finite length, while the one obtained in Theorem \ref{theorem_main1} may be infinite. 
\end{remark}

\begin{remark}
The parameter $L$ in assumption [B1] can be thought of, in terms of applications, as the \emph{order of averaging}. In celestial mechanics models, for example, there are often fast angles (denoted here by $\theta$) that can be averaged out of the Hamiltonian function up to terms of order $\epsilon^L$ for any $L \geq 0$. As such, the conjugate momenta (denoted here by $\xi$) are constant up to terms of order $\epsilon^L$. In this way, the parameter $L$ can typically be chosen in applications to be as large as is required. 
\end{remark}

Theorem \ref{theorem_main2} provides a shadowing argument along the invariant manifolds of a normally hyperbolic invariant cylincer relying on  long sequences of almost invariant leaves of a foliation. Often in Arnold diffusion results, one  wants only to fix the initial and final points and not the whole sequence of leaves. Such statement in the following corollary which is a direct consequence of the assumed hypotheses and Theorem \ref{theorem_main2}.

\begin{corollary}
Fix $\eta>0$, constants $\rho$, $\tau$ and $\upsilon$ satisfying \eqref{eq_kavgcondition}   and $p^*_{\mathrm{ini}}$, $p^*_{\mathrm{fin}}\in [0,1]^n$. Then, for any $L\geq 1$ large enough and $\epsilon>0$ small enough there exists $\xi^*_1 \in \mathrm{Int} \left([0,1]^{\ell_2} \right)$, $N\in\mathbb{N}$ and 
 $\{p_k^*\}_{k=1}^N$ such that 
\begin{itemize} 
\item  $p_1^*=p^*_{\mathrm{ini}}$, $p_N^*=p^*_{\mathrm{fin}}$,
\item $S \left( \L_j \cap U ; \xi^*_1 \right) \cap \L_{j+1} \neq \emptyset$ and $S \left( \L_j \cap U; \xi^*_1 \right)$ is transverse to $\L_{j+1}$, where $\L_j = \L (p_j^*)$
\item The distance between $\L_j$ and $\L_{j+1}$ is of order $\epsilon^{\upsilon}$ for each $j$.
\end{itemize}
Moreover, there are $\xi^*_2, \ldots, \xi^*_N \in [0,1]^{\ell_2}$,  $w_1, \ldots, w_N \in \widetilde{M}$   such that the $\xi$ component of $w_1$ is $\xi^*_1$ and natural numbers and $n_i \in \mathbb{N}$ satisfying
and 
\[
 n_1\leq n_2\leq \ldots \leq n_N\lesssim \epsilon^{- \rho - \tau - \upsilon}.
\]
such that 
\[
w_{i+1} = \Psi^{n_i} (w_i),\qquad \text{ and }\qquad 
d \left(w_i, \widetilde{\L}_i \right) < \eta\]
where  $\widetilde{\L}_j = \widetilde{\L} \left(p^*_j, \xi^*_j \right)$.
%
%
\end{corollary}

\subsection{Heuristic Description of the Proof of Theorems \ref{theorem_main1} and \ref{theorem_main2}}

The key idea of the proof is the construction of a sequence of correctly aligned windows in a neighbourhood of the normally hyperbolic invariant cylinder and homoclinic channel. A window is a product of two rectangles, with each boundary component identified as belonging either to an \emph{entry set} or to an \emph{exit set}. Whether a boundary component belongs to the entry set or the exit set is a free choice; indeed, one can even choose the entry set to be empty, and the exit set to be the entire topological boundary of the window, or vice versa. Informally, two windows $W_1$ and $W_2$ are correctly aligned under a map $f$ if $f(W_1)$ and $W_2$ fully overlap in such a way that the image of the exit set of $W_1$ under $f$ does not intersect $W_2$, and the entry set of $W_2$ does not intersect $f(W_1)$. The crux of this idea is that, if we have a (finite, infinite, or even doubly infinite) sequence of windows $\{ W_n \}$ such that $W_n$ is correctly aligned with $W_{n+1}$ under $f$ for each $n$, then there is a trajectory $\{ x_n \}$ of $f$ passing through this sequence of windows, in the sense that $f( x_n) = x_{n+1}$ and $x_n \in W_n$ for each $n$ (see Section \ref{sec_caw} for the formal definition and references). 

The strategy of the proof, therefore, is to construct explicitly such a sequence of correctly aligned windows. This requires a suitable coordinate system in a neighbourhood of the normally hyperbolic manifold $\Lambda$ (see Section \ref{section_linearisedcoords}) in which the map takes a particular form that allows us to see the twist condition. First we show how to construct a ``short sequence'' of correctly aligned windows (see Section \ref{sec_shortsequence}), beginning in a neighbourhood of a point $x_n$ in the homoclinic channel, approaching a point $y \in \Lambda$ along the stable manifold $W^s(\Lambda)$, moving around $\Lambda$ for a (potentially large, depending on the order of the twist condition) number of iterates, and departing along the unstable manifold $W^u (\Lambda)$ towards another point $x_{n+1}$ in the homoclinic channel. This part of the construction uses only normal hyperbolicity and the twist condition [A2]. 

The next step is to show that, given two such short sequences of correctly aligned windows, we can combine them at a homoclinic point $x_n$ to obtain longer sequences (see Section \ref{sec_coordtransfhomoclinic}). The difficulty here is that the windows at the homoclinic point $x_n$ are expressed in different coordinates: one system of coordinates is obtained by iterating the coordinates near $\Lambda$ forward along the unstable manifold $W^u(\Lambda)$, and the other by iterating the coordinates near $\Lambda$ backward along the stable manifold $W^s (\Lambda)$. In order to guarantee that the windows near $x_n$ are correctly aligned, we need to obtain estimates on the coordinate transformation between these two systems. This step uses: [A1], the transversality of the stable and unstable manifolds at the homoclinic point $x_n$; and [A3], the transversality of images under the scattering map of leaves of the foliation of $\Lambda$.

The final part of the proof of Theorem \ref{theorem_main1} consists in choosing the aspect ratios (i.e. the size of the constituent rectangles) of each window in the sequence in order to guarantee that the sequence can be continued indefinitely. This is done in Section \ref{section_proof1_part3}.   

In order to prove Theorem \ref{theorem_main2}, we show that under conditions [B1] and [B2], the conditions [A1-3] of Theorem \ref{theorem_main1} are satisfied by a truncated version of the map. This allows us to consider the sequence of windows $W_n$ in $M$ constructed in the proof of Theorem \ref{theorem_main1}; we then extend these windows to the extended phase space $\widetilde{M}=M \times \Sigma$ by taking the product of $W_n$ with two new rectangles $\Theta_n, \Xi_n$ where $\Theta_n \times \Xi_n \subset \Sigma$. In the $\Theta_n, \Xi_n$ directions, we choose the exit set to be empty and the entry set to be the entire topological boundary, and we just increase the size of the rectangles $\Theta_n, \Xi_n$ at each step. This guarantees the correct alignment of the windows $W_n \times \Theta_n \times \Xi_n$ at each step under iterates of the truncated map. Finally we show that the error terms of the full map do not spoil the correct alignment; this is true for finite sequences of windows, and it is not clear that it can be extended to infinite sequences. 

The structure of the paper is as follows. In Section \ref{sec_caw} we define windows, and what it means for them to be correctly aligned, and we state several necessary theorems regarding correct alignment. In Section \ref{sec_coordsandestimates} we establish a suitable system of coordinates in a neighbourhood of $\Lambda$, we show how it can be iterated along the stable and unstable manifolds to a homoclinic point, and we obtain estimates on the map in these coordinates. In Section \ref{sec_proofthm1} we prove Theorem \ref{theorem_main1}, and finally Theorem \ref{theorem_main2} is proved in Section \ref{sec_proofthm2}.

\section{Correctly Aligned Windows}\label{sec_caw}

We follow the exposition in \cite{gidea2006topological} (itself based on \cite{gidea2003topologically,gidea2004symbolic,gidea2004covering,zgliczynski2004covering}; see also the appendix of \cite{gidea2020general}), which elaborates on ideas introduced in \cite{easton1978homoclinic,easton1981orbit,easton1979homoclinic}.  

\subsection{Definitions and Main Ideas}

Let $M$ be a manifold of dimension $m$. A window is a subset of $M$ that is a product of $C^0$ rectangles.
\begin{definition} 
Let $m_1, m_2 \in \mathbb{N}_0$ such that $m_1+m_2=m$. 
\begin{itemize}
\item
A set $W \subset M$ is an \emph{$(m_1, m_2)$ window} if there is an open neighbourhood $V$ of $W$ in $M$, an open neighbourhood $\widehat{V}$ of $[0,1]^{m_1} \times [0,1]^{m_2}$ in $\mathbb{R}^m$, and a homeomorphism $\chi :  \widehat{V} \to  V$ such that
\begin{equation}
W = \chi \left( [0,1]^{m_1} \times [0,1]^{m_2} \right). 
\end{equation}
\item
Moreover there is a choice of \emph{entry set}
\begin{equation} \label{eq_entrysetdef}
W^+ = \chi \left( [0,1]^{m_1} \times \partial [0,1]^{m_2} \right)
\end{equation}
and \emph{exit set}
\begin{equation} \label{eq_exitsetdef}
W^- = \chi \left( \partial [0,1]^{m_1} \times  [0,1]^{m_2} \right).
\end{equation}
\end{itemize}
\end{definition}
\begin{remark}
We say that there is a \emph{choice} of entry and exit sets as we could equally have chosen \eqref{eq_exitsetdef} to be the entry set and \eqref{eq_entrysetdef} to be the exit set. In practice, every time we define a window we explicitly state our choice of its entry and exit sets, but we use the definitions \eqref{eq_entrysetdef} and \eqref{eq_exitsetdef} for the purposes of this exposition. 
\end{remark}
\begin{definition} \label{definition_caw}
For $j=1,2$ let $W_j \subset M$ be an $(m_1, m_2)$ window with parametrisation $\chi_j : \widehat{V}_j \to V_j$, and let $f \in C^0(M,M)$ such that $f \left( V_1 \right) \subseteq V_2$. Let 
\begin{equation}
\hat{f} = \chi_2^{-1} \circ f \circ \chi_1 : \widehat{V}_1 \longrightarrow \widehat{V}_2.
\end{equation}
Then $W_1$ is \emph{correctly aligned with $W_2$ under $f$} if there is a homotopy
\begin{equation}
H : [0,1] \times \widehat{V}_1 \longrightarrow \widehat{V}_2
\end{equation}
such that:
\begin{enumerate}
\item
We have
\begin{equation}
H (0, \cdot ) = \hat{f},
\end{equation}
\begin{equation}
H \left( [0,1], \chi_1^{-1} \left( W_1^- \right) \right) \cap \chi_2^{-1} \left( W_2 \right) =\emptyset,
\end{equation}
\begin{equation}
H \left( [0,1], \chi_1^{-1} \left( W_1 \right) \right) \cap \chi_2^{-1} \left( W_2^+ \right) =\emptyset. 
\end{equation}
\item
If $m_1=0$ then $f(W_1) \subset \mathrm{Int} (W_2)$. If $m_1>0$ then there is $y \in [0,1]^{m_2}$ such that the map $A_y : [0,1]^{m_1} \longrightarrow \mathbb{R}^{m_1}$ defined by
\begin{equation}
A_y (x) = \pi_1 \left( H (1, (x,y)) \right)
\end{equation}
satisfies
\begin{equation}
A_y \left( \partial [0,1]^{m_1} \right) \subset \mathbb{R}^{m_1} \setminus [0,1]^{m_1}, \quad \deg \left( A_y, 0 \right) \neq 0
\end{equation}
where $\pi_1 : \mathbb{R}^{m_1} \times \mathbb{R}^{m_2} \to \mathbb{R}^{m_1}$ is the canonical projection onto the first component, and $\deg \left( A_y, 0 \right)$ is the Brouwer degree of the map $A_y$ at 0. 
\end{enumerate}
In part 2 of the definition, whenever $m_1>0$, if instead we have that there is a linear map $A : \mathbb{R}^{m_1} \to \mathbb{R}^{m_1}$ such that $H(1,(x,y))=(Ax,0)$ for all $x \in [0,1]^{m_1}, y \in [0,1]^{m_2}$, and $A \left( \partial [0,1]^{m_1} \right) \subset \mathbb{R}^{m_1} \setminus [0,1]^{m_1}$, then we say that $W_1$ is \emph{linearly correctly aligned with $W_2$ under $f$. }
\end{definition}

\begin{remark}
Observe that the property of two windows being linearly correctly aligned is stronger than simply being correctly aligned, in the sense that linearly correct alignment implies correct alignment (see Proposition 2 of \cite{gidea2006topological}). This becomes useful when we consider products of windows (see Section \ref{section_productsofwindows}).
\end{remark}

The following result (Corollary 12 of \cite{zgliczynski2004covering}), whimsically characterised as the property that `one can see through a sequence of correctly aligned windows', is the main point of this technique.
\begin{theorem} \label{theorem_onecansee}
Let $\left\{ W_i \right\}_{i \in \mathbb{Z}}$ be a collection of $(m_1, m_2)$ windows in $M$, and $\left\{ f_i \right\}_{i \in \mathbb{Z}} \subset C^0 (M,M)$ a collection of continuous mappings such that $W_i$ is correctly aligned with $W_{i+1}$ under $f_i$ for each $i \in \mathbb{Z}$. Then we can find $\left\{ z_i \right\}_{i \in \mathbb{Z}} \subset M$ such that
\begin{equation}
z_i \in W_i, \quad f_i (z_i) = z_{i+1}
\end{equation}
for all $i \in \mathbb{Z}$.
\end{theorem}

The property of two windows being correctly aligned under a map is stable under perturbation in the sense of the next result (Theorem 13 of \cite{zgliczynski2004covering}).
\begin{theorem}\label{theorem_cawstable}
Suppose $W_1, W_2 \subset M$ are $(m_1, m_2)$ windows such that $W_1$ is correctly aligned with $W_2$ under a map $f \in C^0(M,M)$. Then there is an open neighbourhood $\U$ of $f$ in $C^0(M,M)$ with respect to the compact-open topology such that $W_1$ is correctly aligned with $W_2$ under $\tilde{f}$ for all $\tilde{f} \in \U$. 
\end{theorem}

\subsection{Products of Windows} \label{section_productsofwindows}

Suppose $M_j$ is a manifold of dimension $k_j$ for $j=1,2$, and suppose the manifold
\begin{equation} \label{eq_productmanifold}
M = M_1 \times M_2
\end{equation}
is equipped with the product topology. For $j=1,2$ let $W_j \subset M_j$ be an $(m_j,n_j)$ window with parametrisation $\chi_j : \widehat{V}_j \to V_j$ where $m_j,n_j \in \mathbb{N}_0$ such that $k_j=m_j+n_j$. Suppose moreover
\begin{equation}
W_j^+ = \chi_j \left( [0,1]^{m_j} \times \partial [0,1]^{n_j} \right), \quad W_j^- = \chi_j \left( \partial [0,1]^{m_j} \times [0,1]^{n_j} \right).
\end{equation}
Now, the sets $\widehat{V} = \widehat{V}_1 \times \widehat{V}_2 \subset \mathbb{R}^{k_1} \times \mathbb{R}^{k_2}$ and $V = V_1 \times V_2 \subset M$ are open in the product topology, and $W=W_1 \times W_2 \subset V$. Define $\chi : \widehat{V} \to V$ by
\begin{equation}
\chi (x,y) = \left( \chi_1 (x), \chi_2 (y) \right)
\end{equation}
where $x \in \widehat{V}_1$, $y \in \widehat{V}_2$. Define the entry and exit sets of $W$ to be
\begin{equation} \label{eq_productentryexit}
\begin{dcases}
W^+ = \left( W_1^+ \times W_2 \right) \cup \left( W_1 \times W_2^+ \right) \\
W^- = \left( W_1^- \times W_2 \right) \cup \left( W_1 \times W_2^- \right).
\end{dcases}
\end{equation}
It can thus be seen that $W$ is an $(m_1 + m_2, n_1 + n_2)$ window in $M$ with parametrisation $\chi : \widehat{V} \to V$. In this case we say that $W$ is the product of the windows $W_1, W_2$, and we write $W = W_1 \times W_2$. 

Suppose now that $f \in C^0 (M,M)$, and for $(x,y) \in M_1 \times M_2$ we write
\begin{equation}
f(x,y) = \left( f_1 (x,y), f_2 (x,y) \right)
\end{equation}
where $f_j (x,y) \in M_j$. The following result was proved in \cite{gidea2006topological}.

\begin{theorem} \label{theorem_cawproduct}
Let $W = W_1 \times W_2$, $\widetilde{W} = \widetilde{W}_1 \times \widetilde{W}_2$ be $(m_1 + m_2, n_1 + n_2)$ windows in $M = M_1 \times M_2$. Suppose
\begin{enumerate}[(i)]
\item
$W_1$ is linearly correctly aligned with $\widetilde{W}_1$ under $f_1 ( \cdot, y)$ for each $y \in M_2$; and
\item
$W_2$ is linearly correctly aligned with $\widetilde{W}_2$ under $f_2 ( x, \cdot)$ for each $x \in M_1$.
\end{enumerate}
Then $W$ is correctly aligned with $\widetilde{W}$ under $f$. 
\end{theorem} 

\begin{remark}
In general, the manifolds $M$ in our work have the structure \eqref{eq_productmanifold} only locally. However this is enough to apply Theorem \ref{theorem_cawproduct}, since two windows being correctly aligned under a map is a local property.
\end{remark}

\section{Coordinates and Estimates}\label{sec_coordsandestimates}

\subsection{A Suitable System of Coordinates}\label{section_linearisedcoords}

Suppose we are in the setting of Theorem \ref{theorem_main1}, so $M$ is a $C^r$ manifold of dimesion $2(m+n)$, and $F \in \Diff^r (M)$ has a normally hyperbolic invariant manifold $\Lambda \simeq \mathbb{T}^n \times [0,1]^n$ in $M$ satisfying [A1-3]. In order to construct correctly aligned windows, we need estimates for the map in a neighbourhood of the invariant manifolds $W^{s,u}(\Lambda)$. This requires an appropriate system of coordinates in which the twist property of the inner map $F|_{\Lambda}$ is apparent; Fenichel coordinates (described below; see \cite{jones2009generalized}) provide a starting point for the coordinate transformation. However, when we express the map $F$ in Fenichel coordinates, there are error terms that complicate the estimates. We therefore seek a further coordinate transformation in which these error terms can be ignored. There is an analogue of the Hartman-Grobman Theorem for normally hyperbolic invariant manifolds, which says that there is a neighbourhood of $\Lambda$ in which $F$ is topologically conjugate to its so-called \emph{normal map} \cite{pugh1970linearization}. These coordinates would be ideal for obtaining \emph{local} estimates in a neighbourhood of the cylinder (as in Section \ref{sec_shortsequence}), but later we need to analyse how the images of this coordinate chart fit together at a homoclinic point if we iterate it backwards (resp. forwards) along the stable (resp. unstable) manifold of $\Lambda$ (see Section \ref{sec_coordtransfhomoclinic}). This analysis requires a continuous second derivative; since the coordinates provided by \cite{pugh1970linearization} are purely topological, we instead follow Section 5.1 of \cite{gidea2006topological} to construct a suitable system of coordinates that are as smooth as required. 

Denote by $(q,p)$ the coordinates on $\Lambda$. Assuming $\Lambda$ is a $C^r$ submanifold of $M$, and the spectral gap is of size $r$ we can introduce, as in \cite{jones2009generalized}, $C^{r-1}$ coordinates $(s,u,q,p)$ in a neighbourhood $U_0$ of $\Lambda$, where $s,u$ belong to a neighbourhood of the origin in $\mathbb{R}^m$ such that:
\begin{itemize}
\item
$W^s_{\mathrm{loc}} (\Lambda) = \{u=0\}$;
\item
$W^u_{\mathrm{loc}} (\Lambda) = \{s=0\}$;
\item
$W^s_{\mathrm{loc}} (q_0,p_0) = \{u=0, \, (q,p)=(q_0,p_0) \}$; and
\item
$W^u_{\mathrm{loc}} (q_0,p_0) = \{s=0, \, (q,p)=(q_0,p_0) \}$.
\end{itemize}
In these coordinates the map $F$ takes the form
\begin{equation}
F: 
\begin{cases}
(\bar{q}, \bar{p}) \! \! \! \! \!&= f(q,p) + N_c(s,u,q,p)\\
\hphantom{(,)} \bar{s} \hphantom{\bar{p}} \! \! \! \! \! &= A_s(q,p) \, s + N_s (s,u,q,p) \\
\hphantom{(,)} \bar{u} \hphantom{\bar{p}} \! \! \! \! \! &= A_u(q,p) \, u + N_u (s,u,q,p)
\end{cases}
\end{equation}
where $f = F|_{\Lambda}$, $A_s(q,p) = D F (q,p)|_{E^s}$, $A_u(q,p) = D F (q,p)|_{E^u}$, and
\begin{equation}
N_c(0,u,q,p)=0=N_c(s,0,q,p), \quad N_s(0,u,q,p)=0=N_u(s,0,q,p). 
\end{equation}

In this paper we study the dynamics only in a small neighbourhood of the invariant manifolds $W^{s,u} (\Lambda)$. Therefore we can replace $F$ by a map $\widetilde{F}_a$ defined as follows. Choose a function $\psi \in C^{\infty} (\mathbb{R})$ such that $\psi (x) = 1$ if $|x| \leq 1$ and $\psi (x)=0$ if $|x| \geq 2$, and let $\psi_a(x) = \psi(ax)$. We then define the map
\begin{equation}\label{eq_lincoordsfatilde}
\widetilde{F}_a: 
\begin{cases}
(\bar{q}, \bar{p}) \! \! \! \! \!&= f(q,p) + \psi_a \left(s^2+u^2 \right) N_c(s,u,q,p)\\
\hphantom{(,)} \bar{s} \hphantom{\bar{p}} \! \! \! \! \! &= A_s(q,p) s + \psi_a \left(s^2+u^2 \right) N_s (s,u,q,p) \\
\hphantom{(,)} \bar{u} \hphantom{\bar{p}} \! \! \! \! \! &= A_u(q,p) u + \psi_a \left(s^2+u^2 \right) N_u (s,u,q,p).
\end{cases}
\end{equation}
Note that the map $\widetilde{F}_a$ is only $C^{r-1}$, since the coordinates $(s,u,q,p)$ are $C^{r-1}$. Moreover the normally hyperbolic invariant manifold $\Lambda$ is a $C^{r-1}$ submanifold, and the spectral gap remains of size $r$. The map $\widetilde{F}_a$ agrees with $F$ whenever $s^2 +u^2 \leq a^{-1}$. Furthermore the error terms of $\widetilde{F}_a$ are uniformly small in $C^0$, and can be made as small as necessary by increasing $a$. We thus guarantee, by choosing $a$ large enough, that the map $\widetilde{F}_a$ is globally defined on $\mathbb{T}^n \times [0,1]^n \times \mathbb{R}^{2m}$, since we have effectively set the error terms equal to 0 outside the domain of definition of $N_c, N_s, N_u$. Fix some sufficiently large $a$ once and for all, and define $\Phi=\widetilde{F}_a$. 

Since $\Phi$ is defined globally, the local unstable manifold $W^u_{\mathrm{loc}}(\Lambda)$ extends to a global unstable manifold $\widetilde{W}^u(\Lambda) = \bigcup_{n=0}^{\infty} \Phi^n \left( W^u_{\mathrm{loc}} \left( \Lambda \right) \right)$. On $\widetilde{W}^u(\Lambda)$ we have the strong unstable foliation, each leaf $\widetilde{W}^u(q,p)$ of which is uniquely determined by a point $(q,p) \in \Lambda$ in the usual way (see Section \ref{sec_prelimdefns} for definitions). Since $\widetilde{W}^u \left( \Lambda \right) = \{ s =0 \}$, the variables $(u,q,p)$ define coordinates on $\widetilde{W}^u \left( \Lambda \right)$. Now, it follows from \eqref{eq_lincoordsfatilde} that the stable manifold $\W^s = W^s \left( \widetilde{W}^u \left( \Lambda \right) \right)$ of the unstable manifold $\widetilde{W}^u \left( \Lambda \right)$ is the entire phase space $\Lambda \times \mathbb{R}^{2m}$. Moreover $\W^s$ admits a $C^{r-2}$ foliation by the leaves $\W^s (u,q,p) = W^s(0,u,q,p)$. Notice that the restriction of $\Phi$ to $\widetilde{W}^u \left( \Lambda \right)$ is 
\begin{equation}
\left. \Phi \right|_{\widetilde{W}^u \left( \Lambda \right)}: 
\begin{cases}
(\bar{q}, \bar{p}) \! \! \! \! \!&= f(q,p) \\
\hphantom{(,)} \bar{u} \hphantom{\bar{p}} \! \! \! \! \! &= A_u(q,p) u + \psi_a \left(u^2 \right) N_u (0,u,q,p).
\end{cases}
\end{equation}
We can now take a coordinate $s'$ on each $\W^s (u,q,p)$ so that $(s',u,q,p)$ are $C^{r-2}$ coordinates in a neighbourhood of $\Lambda$, and
\begin{equation}\label{eq_mapinappcoords}
\Phi: 
\begin{cases}
(\bar{q}, \bar{p}) \! \! \! \! \!&= f(q,p) \\
\hphantom{(,)} \bar{s}' \hphantom{\bar{p}} \! \! \! \! \! &= A_s(q,p) s' +  \widetilde{N}_s (s',u,q,p) \\
\hphantom{(,)} \bar{u} \hphantom{\bar{p}} \! \! \! \! \! &= A_u(q,p) u + \psi_a \left(u^2 \right) N_u (0,u,q,p)
\end{cases}
\end{equation}
where $\widetilde{N}_s$ is globally small, and $\widetilde{N}_s(0,u,q,p)=0$. Since $F$ agrees with $\Phi$ in a neighbourhood of $\Lambda$, it also takes the form \eqref{eq_mapinappcoords} in $(s',u,q,p)$ coordinates. From now on we drop the prime notation and write $(s,u,q,p)$ for this system of coordinates.

We have thus constructed $C^{r-2}$ coordinates in a neighbourhood of $\Lambda$ in which there are no error terms in the central directions, and the error terms in the hyperbolic directions are small: let $\delta_s, \delta_u>0$ such that
\begin{equation}
\sup_{(s,u,q,p)} \left\| \widetilde{N}_s (s,u,q,p) \right\| \leq \delta_s \| s \|, \quad \sup_{(u,q,p)} \left\| \psi_a \left( u^2 \right)N_u (0,u,q,p) \right\| \leq \delta_u \| u \|
\end{equation}
where we denote by $\| x \| = \max_{i=1, \ldots, m} |x_i|$ the maximum norm. Since $a$ is sufficiently large, $\delta_s, \delta_u$ are sufficiently small. Due to the normal hyperbolicity estimates \eqref{eq_normalhyperbolicitydef}, we find that
\begin{equation}
\tilde{\lambda}_- \| s \| \leq \left\| A_s (q,p) s + \widetilde{N}_s (s,u,q,p) \right\| \leq \tilde{\lambda}_+ \| s \|
\end{equation}
\begin{equation}
\tilde{\mu}_- \| u \| \leq \left\| A_u (q,p) u + \psi_a \left( u^2 \right) N(0,u,q,p) \vphantom{\widetilde{N}_s} \right\| \leq \tilde{\mu}_+ \| u \|
\end{equation}
 where $\tilde{\lambda}_{\pm} = \lambda_{\pm} \pm \delta_s$ and $\tilde{\mu}_{\pm} = \mu_{\pm} \pm \delta_u$. Since $\delta_{s,u}$ are sufficiently small, we still have 
 \begin{equation}
 0 < \tilde{\lambda}_- < \tilde{\lambda}_+ < \lambda_0^r < 1 < \mu_0^r < \tilde{\mu}_- < \tilde{\mu}_+.
 \end{equation}
 In this way we can use the linear estimates \eqref{eq_normalhyperbolicitydef} for the nonlinear map $F$ in the hyperbolic directions. We now drop the tilde notation, and simply write $\lambda_{\pm}, \mu_{\pm}$ for these adjusted hyperbolicity parameters. The following section deals with estimates in the central directions. 

\begin{remark}
Since the coordinates $(s,u,q,p)$ are $C^{r-2}$, and since we require the coordinate transformation near a homoclinic point (see Section \ref{sec_coordtransfhomoclinic}) to have two continuous derivatives, we take $r \geq 4$ in the statement of Theorem \ref{theorem_main1}. 
\end{remark}

We also require a system of coordinates near the homoclinic channel. Denote by $\U$ the neighbourhood of $\Lambda$ in which we have the $(s,u,q,p)$ coordinates, by $\widehat{\U}$ the neighbourhood of $\mathbb{T}^n \times [0,1]^n \times \{ 0 \}$ in $\mathbb{T}^n \times [0,1]^n \times \mathbb{R}^{2m}$ to which the $(s,u,q,p)$ variables belong, and by $h: \U \to \widehat{\U}$ the $C^2$ coordinate transformation constructed above. Let $x \in \Gamma \subset \left( W^s (\Lambda) \pitchfork W^u (\Lambda) \right) \setminus \Lambda$ be a transverse homoclinic point. Then there are $N_{\pm} \in \mathbb{N}$ such that
\begin{equation}
F^{N_+} (x) \in W^s_{\mathrm{loc}}(\Lambda) \cap \U, \quad F^{-N_-} (x) \in W^u_{\mathrm{loc}}(\Lambda) \cap \U.
\end{equation}
Choose a neighbourhood $V$ of $x$ in $M$ such that $F^{N_+}(V) \subset \U$ and $F^{-N_-}(V) \subset \U$. We define two coordinate systems $(s^+,u^+,q^+,p^+)$ and $(s^-,u^-,q^-,p^-)$ on $V$ via the diffeomorphisms $h^{\pm} :V \to \widehat{\U}$ defined by
\begin{equation}
h^+ = \Phi^{-N_+} \circ h \circ F^{N_+}, \quad h^- = \Phi^{N_-} \circ h \circ F^{-N_-}.
\end{equation}
From $(s^+,u^+,q^+,p^+)$ coordinates to $(s,u,q,p)$ coordinates, the map $F^{N_+}$ acts as $\Phi^{N_+}$, in the sense that
\begin{equation}
h \circ F^{N_+} \circ \left( h^+ \right)^{-1} = \Phi^{N_+}. 
\end{equation}
Similarly from $(s^-,u^-,q^-,p^-)$ coordinates to $(s,u,q,p)$ coordinates, the map $F^{-N_-}$ acts as $\Phi^{-N_-}$, in the sense that
\begin{equation}
h \circ F^{-N_-} \circ \left( h^- \right)^{-1} = \Phi^{-N_-}.
\end{equation}

\subsection{Estimates on the Shearing of a Window by the Twist Map}

Recall the inner map $f = F|_{\Lambda}$ is given by \eqref{eq_inttwistmap}, and we assume moreover the twist condition \eqref{eq_twistcondition}, where $\tilde{T}_- = \epsilon^{\tau} T_-$ for some $\tau \in \mathbb{N}_0$. We use the maximum norm: 
\begin{equation}\label{eq_maxnormdef}
\| x \| = \max_{i=1, \ldots, n} |x_i|.
\end{equation}
Fix $R>0$ with the following property. For any $p^0,p \in [0,1]^n$ we can write
\begin{equation}
g(p) = g(p^0) + D g (p^0) (p-p^0) + R_* (p^0,p)
\end{equation}
where $R_* \left(p^0,p \right)$ is a remainder term of order $O \left(\left\| p - p^0 \right\|^2 \right)$. Then the constant $R$ is chosen so that
\begin{equation}\label{eq_constr}
\left\| R_* \left(p^0,p \right) \right\| \leq R \left\| p - p^0 \right\|^2
\end{equation}
for all $p^0,p \in [0,1]^n$. 

Define a window $W = [Q \times P] \subset \Lambda$ where 
\begin{equation}
Q = [a, a + \gamma]^n \subset \mathbb{T}^n, \quad P = [b, b + \delta]^n \subset [0,1]^n. 
\end{equation}
Choose the exit set $W^-$ to be, say,
\begin{equation}
W^- = Q \times \partial P.
\end{equation}
The following estimates apply equally if the exit set is chosen to be in the $q$-direction. 

For each $j=1, \ldots, n$ define
\begin{equation}
B^0_j = [b, b+ \delta]^{j-1} \times \{ b \} \times [b, b+ \delta]^{n-j}, \quad B^1_j = [b, b+ \delta]^{j-1} \times \{ b+ \delta \} \times [b, b+ \delta]^{n-j}
\end{equation}
and let
\begin{equation}
E^0_j = Q \times B^0_j, \quad E^1_j = Q \times B^1_j
\end{equation}
be the corresponding components of the exit set, in the sense that
\begin{equation}
W^- = \bigcup_{j=1}^n \left( E^0_j \cup E^1_j \right).
\end{equation}
Suppose now that $N \in \mathbb{Z}$, and suppose we take an iterate $f^N (W)$ of $W$. Define
\begin{equation}
\Delta_j^N = \min_{(q^0,p^0) \in E^0_j} \min_{(q^1,p^1) \in E^1_j} \left\| \pi \circ f^N (q^1, p^1) - \pi \circ f^N (q^0, p^0) \right\|
\end{equation}
and
\begin{equation}
\Omega^N = \max_{(q^0,p^0),(q^1,p^1) \in W} \left\| \pi \circ f^N (q^1, p^1) - \pi \circ f^N (q^0, p^0) \right\|
\end{equation}
where $\pi : \mathbb{T}^n \times [0,1]^n \to \mathbb{T}^n$ is the canonical projection. The following lemma gives lower and upper bounds on the shearing of the window $W$ under $f^N$. 

\begin{lemma}\label{lemma_shearingestimates}
There is a positive constant $C$ such that
\begin{equation}\label{eq_shearboundlower}
\Delta_j^N \geq \epsilon^{\tau} |N| T_- \delta - |N| R \delta^2 - \gamma - C |N|^2 \epsilon^k
\end{equation}
and
\begin{equation}\label{eq_shearboundupper}
\Omega^N \leq \gamma + |N| T_+ \delta + C |N|^2 \epsilon^k
\end{equation}
where $\epsilon^{\tau}$ is the order of the twist condition, $\epsilon^k$ is the order of the error terms in the definition \eqref{eq_inttwistmap} of $f$, and the positive constant $R$ is defined by \eqref{eq_constr}. 
\end{lemma}

\begin{proof}
We establish \eqref{eq_shearboundlower} for the case $j=1$ as the remaining cases are analogous. Fix $q^0,q^1 \in Q$ and $p^0_*, p^1_* \in [b, b + \delta]^{n-1}$. Define
\begin{equation}
p^0 = (b, p^0_*), \quad p^1 = (b+ \delta, p^1_*)
\end{equation}
Notice that
\begin{equation}
\| p^1 - p^0 \| = \delta. 
\end{equation}
Write
\begin{equation}
(\bar{q}^j, \bar{p}^j) = f^N (q^j, p^j)
\end{equation}
for $j=0,1$ so that
\begin{equation}
\bar{q}^j = q^j + N g(p^j) + O \left(N^2 \epsilon^k \right), \quad \bar{p}^j = p^j + O(N \epsilon^k). 
\end{equation}
Let $C$ be a uniform upper bound on the terms of order $\epsilon^k$. Then we have
\begin{align}
\left\| \bar{q}^1 - \bar{q}^0 \right\| \geq{}& |N| \left\| g(p^1) - g(p^0) \right\| - \left\| q^1 - q^0 \right\| - C |N|^2 \epsilon^k \\
\geq{}& |N| \left\| D g (p^0) (p^1-p^0) + R_* (p^0,p^1) \right\| - \gamma - C |N|^2 \epsilon^k \\
\geq{}& \epsilon^{\tau} |N| T_- \delta - |N| R \delta^2 - \gamma - C |N|^2 \epsilon^k.
\end{align}

In order to prove \eqref{eq_shearboundupper}, fix $(q^j,p^j) \in W$ and let $(\bar{q}^j, \bar{p}^j) = f^N (q^j, p^j)$ for $j=0,1$. Then
\begin{align}
\left\| \bar{q}^1 - \bar{q}^0 \right\| \leq{}& \left\| q^1 - q^0 \right\| + |N| \left\| g(p^1) - g(p^0) \right\| + C |N|^2 \epsilon^k \\
\leq{}& \gamma + |N| T_+ \delta + C |N|^2 \epsilon^k.
\end{align}
\end{proof}

\section{Proof of Theorem \ref{theorem_main1}} \label{sec_proofthm1}

Denote by $\L_1, \ldots, \L_N$ a sequence of leaves of the foliation of $\Lambda$ with transverse homoclinic connections, as in the statement of Theorem \ref{theorem_main1}. Let $x_i \in \Gamma$ be a point in the homoclinic channel such that $\pi^s (x_i) \in \L_i$ and $\pi^u (x_i) \in \L_{i+1}$. The proof of Theorem \ref{theorem_main1} is divided into 3 parts. In the first part of the proof, we show how to construct a short chain of correctly aligned windows beginning at the homoclinic point $x_i$, moving near to $\L_i$, and ending at the homoclinic point $x_{i+1}$. In the second part of the proof we show that we can continue the sequence of windows across a transverse homoclinic intersection in order to glue together two short sequences. In the final part of the proof, it is shown that we can consistently choose the aspect ratios of the windows in a way that allows us to continue the sequence indefinitely. 

\subsection{Construction of a Short Sequence of Correctly Aligned Windows}\label{sec_shortsequence}

Throughout the proof we use the coordinates $(s,u,q,p)$ constructed in Section \ref{section_linearisedcoords} near the normally hyperbolic invariant manifold $\Lambda$, as well as the coordinates $(s^+,u^+,q^+,p^+)$ and $(s^-,u^-,q^-,p^-)$, defined in Section \ref{section_linearisedcoords}, near the homoclinic channel $\Gamma$. 

This part of the proof is performed in 3 steps. In the first step, we construct a window $W_i$ in $(s^+,u^+,q^+,p^+)$ coordinates centred at $x_i \in \Gamma$, and a window $\widetilde{W}_i$ in $(s,u,q,p)$ coordinates centred at $f^{N_i}(y_i^s)$ where $y_i^s  = \pi^s (x_i)$ such that $W_i$ is correctly aligned with $\widetilde{W}_i$ under $F^{N_i}$ for some $N_i \in \mathbb{N}$. In the second step, we construct a window $\widehat{W}_i$ in $(s,u,q,p)$ coordinates centred at $f^{N_i+K_i}(y_i^s)$ such that $\widetilde{W}_i$ is correctly aligned with $\widehat{W}_i$ under $F^{K_i}$ for some $K_i \in \mathbb{N}$. In the third step, we construct a window $W_i'$ in $(s^-,u^-,q^-,p^-)$ coordinates, centred at the transverse homoclinic point $x_{i+1} \in \Gamma$ such that $\widehat{W}_i$ is correctly aligned with $W_i'$ under $F^{M_i}$ for some $M_i \in \mathbb{N}$. 

In each step, the windows are defined as a product of two constituent windows. We ensure that the windows we construct are correctly aligned by ensuring that the constituent windows are \emph{linearly} correctly aligned (see Definition \ref{definition_caw} and Theorem \ref{theorem_cawproduct}). In the hyperbolic directions (i.e. the $s,u$ variables) we simply use the normal hyperbolicity estimates \eqref{eq_normalhyperbolicitydef}. In the centre directions (i.e. the $q,p$ variables) we use the shearing estimates provided by Lemma \ref{lemma_shearingestimates}. Moreover at each step we state the entry and exit sets of the constituent windows, in which case the entry and exit sets of the product of these windows is given by \eqref{eq_productentryexit}. 

Finally, let us define the \emph{centre} of a window $W$ in $(s,u,q,p)$ coordinates. Suppose we define a window
\begin{equation}
W = \left[ S \times U \right] \times \left[ Q \times P \right]
\end{equation}
where $S,U,Q,P$ is a rectangle in $s,u,q,p$ respectively, and suppose the rectangles have size
\begin{align}
\left| S \right| ={}& \max_{s_1,s_2 \in S} \| s_1 - s_2 \| =  \alpha, \quad \left| U \right| = \max_{u_1,u_2 \in U} \| u_1 - u_2 \| =  \beta, \\
\left| Q \right| ={}& \max_{q_1,q_2 \in Q} \| q_1 - q_2 \| = \gamma, \quad \left| P \right| = \max_{p_1,p_2 \in P} \| p_1 - p_2 \| = \delta,
\end{align}
where the maximum norm $\| . \|$ is defined by \eqref{eq_maxnormdef}. The \emph{centre} of $W$ is the point 
\begin{equation}
c^0 = (s^0, u^0, q^0, p^0)
\end{equation}
such that
\begin{align}
S ={}& \left[ s^0_1 - \frac{\alpha}{2}, s^0_1 + \frac{\alpha}{2} \right] \times \cdots \times \left[ s^0_m - \frac{\alpha}{2}, s^0_m + \frac{\alpha}{2} \right], \\
U ={}& \left[ u^0_1 - \frac{\beta}{2}, u^0_1 + \frac{\beta}{2} \right] \times \cdots \times \left[ u^0_m - \frac{\beta}{2}, u^0_m + \frac{\beta}{2} \right], \\
Q ={}& \left[ q^0_1 - \frac{\gamma}{2}, q^0_1 + \frac{\gamma}{2} \right] \times \cdots \times \left[ q^0_n - \frac{\gamma}{2}, q^0_n + \frac{\gamma}{2} \right], \\
P ={}& \left[ p^0_1 - \frac{\delta}{2}, p^0_1 + \frac{\delta}{2} \right] \times \cdots \times \left[ p^0_n - \frac{\delta}{2}, p^0_n + \frac{\delta}{2} \right]. 
\end{align}

\subsubsection{Step 1}

Define a window $W_i$, centred at the homoclinic point $x_i \in \Gamma$ and given in $(s^+,u^+,q^+,p^+)$ coordinates by 
\begin{equation}
W_i = \left[S_i \times U_i \right] \times \left[Q_i \times P_i \right]
\end{equation}
where $S_i, U_i, Q_i, P_i$ are rectangles in $s^+, u^+, q^+, p^+$ respectively, with sizes
\begin{equation}
|S_i| = \alpha_i, \quad |U_i| = \beta_i, \quad|Q_i| = \gamma_i, \quad |P_i| = \delta_i. 
\end{equation}
We choose the exit sets to be
\begin{equation}
[S_i \times U_i]^- = \hphantom{^-} S_i \times \partial U_i, \quad [Q_i \times P_i]^- = \hphantom{^-} Q_i \times \partial P_i. 
\end{equation}
Now choose $N_i \in \mathbb{N}$ such that $F^{N_i} (x_i) \in W^s_{\mathrm{loc}} (\Lambda) \cap \U$ where $\U$ is the neighbourhood of $\Lambda$ in which $(q,p,s,u)$ coordinates are defined. Let
\begin{equation}
\nu_i = d_{W^s (\Lambda)} (x_i, y_i^s),
\end{equation} 
where $d_{W^s (\Lambda)}$ is the distance measured along the stable manifold, and $y_i^s = \pi^s (x_i)$. Then
\begin{equation}
\nu_i \lambda_-^{N_i} \leq d_{W^s(\Lambda)} (F^{N_i}(x_i), F^{N_i}(y_i^s)) \leq \nu_i \lambda_+^{N_i}
\end{equation}
and moreover by \eqref{eq_normalhyperbolicitydef} and Lemma \ref{lemma_shearingestimates}:
\begin{itemize}
\item
In the $s$-direction, $W_i$ is contracted by $F^{N_i}$ to a size between $\alpha_i \lambda_-^{N_i}$ and $\alpha_i \lambda_+^{N_i}$;
\item
In the $u$-direction, $W_i$ is expanded by $F^{N_i}$ to a size between $\beta_i \mu_-^{N_i}$ and $\beta_i \mu_+^{N_i}$;
\item
In the $q$-direction, $W_i$ is sheared by the twist map to a size at most $\gamma_i + N_i T_+ \delta_i + C N_i \epsilon^k$; and
\item
In the $p$-direction, the change is $\pm C N_i \epsilon^k$. 
\end{itemize}

\begin{figure}[h]
    \centering
    \begin{subfigure}[b]{0.45\textwidth}
        \includegraphics[width=\textwidth]{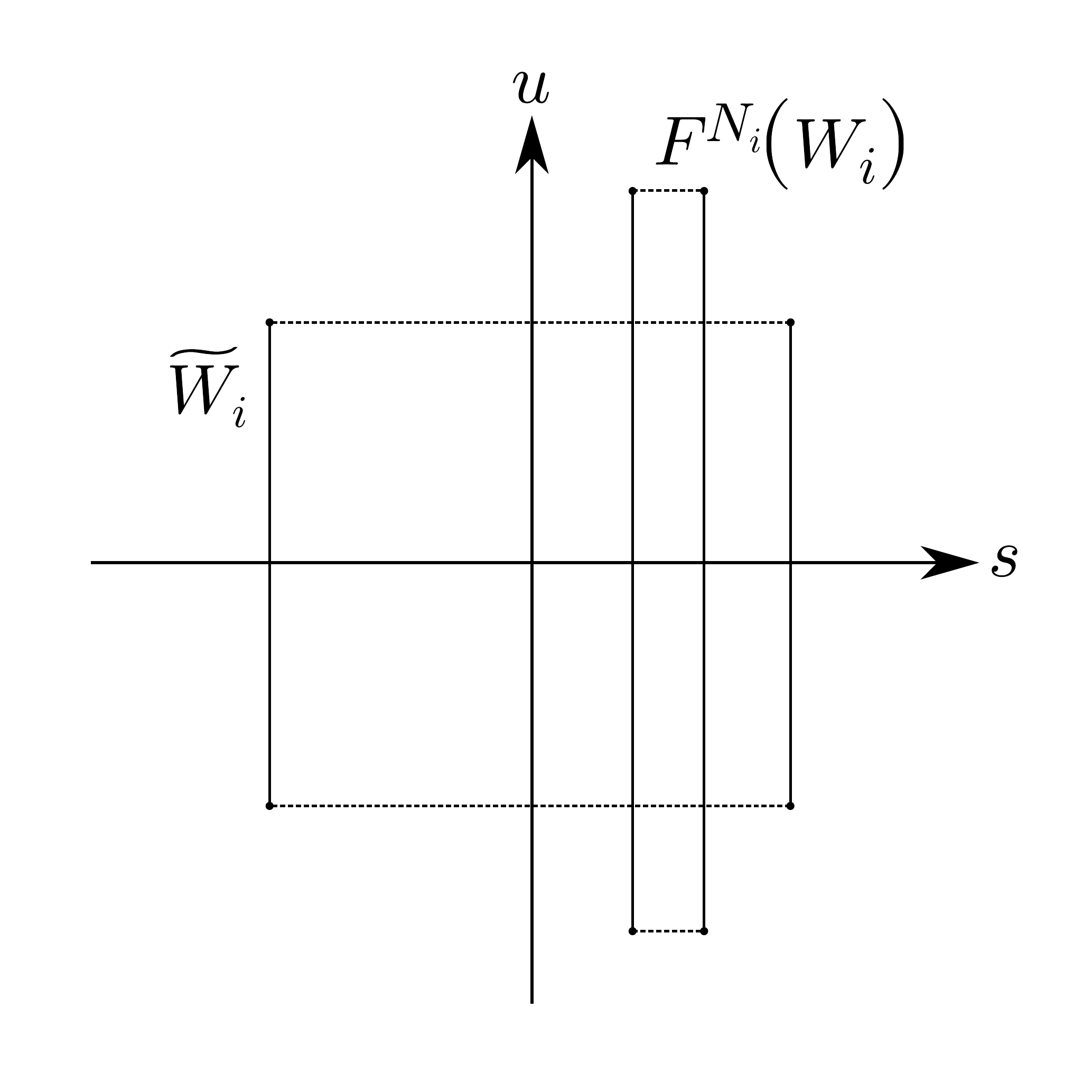}
        \caption{}
        \label{figure_step1}
    \end{subfigure} 
    \begin{subfigure}[b]{0.45\textwidth}
        \includegraphics[width=\textwidth]{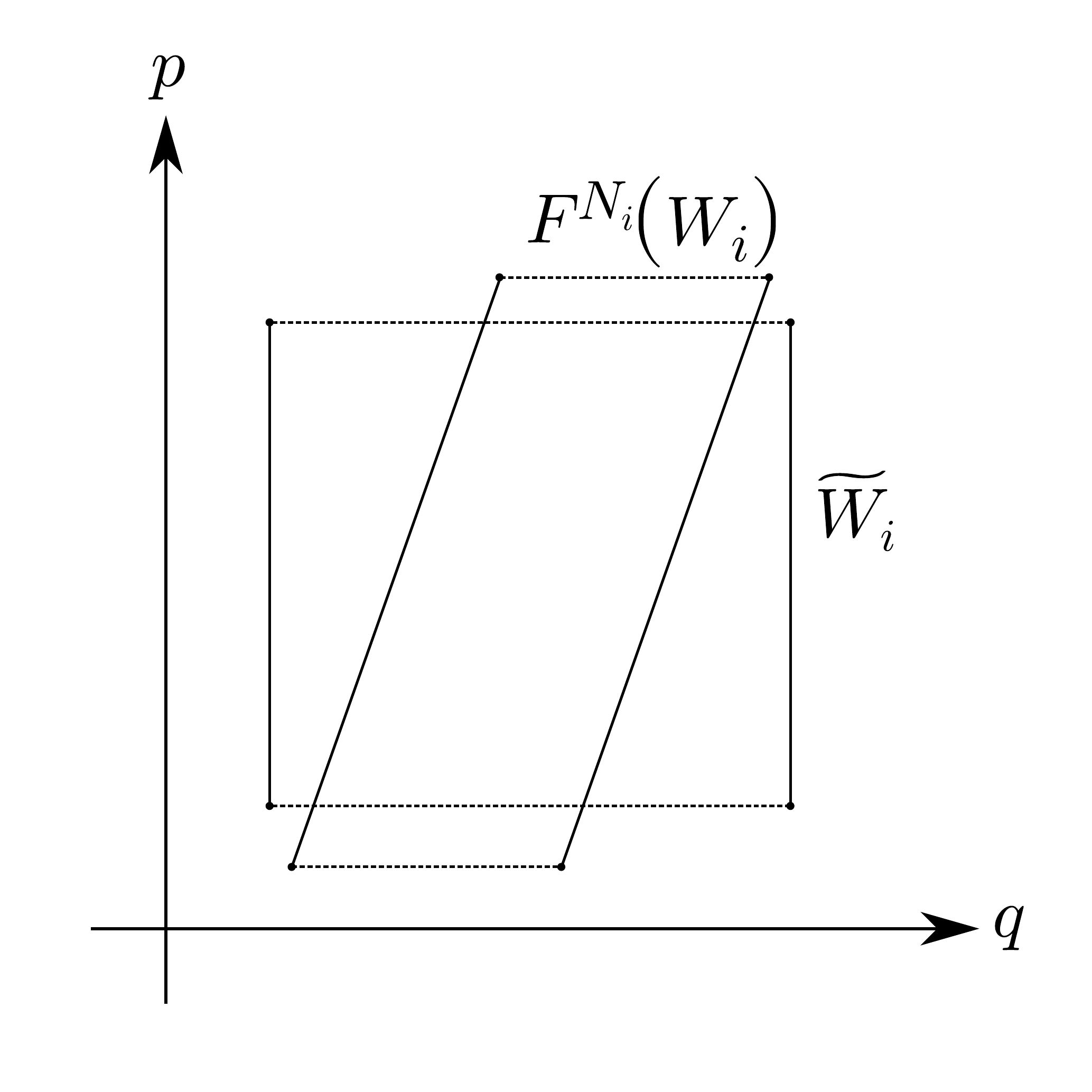}
        \caption{}
        \label{figure_step1}
    \end{subfigure}
    \caption{Step 1. The dashed edges represent exit sets.}\label{figure_step1}
\end{figure}

Define a window $\widetilde{W}_i$ centred at $f^{N_i}(y_i^s) \in \Lambda$ and given in $(s,u,q,p)$ coordinates by 
\begin{equation}
\widetilde{W}_i = \left[\widetilde{S}_i \times \widetilde{U}_i \right] \times \left[\widetilde{Q}_i \times \widetilde{P}_i \right]
\end{equation}
where $\widetilde{S}_i, \widetilde{U}_i, \widetilde{Q}_i, \widetilde{P}_i$ are rectangles in $s,u,q,p$ respectively, and where we choose the exit sets to be 
\begin{equation}
\left[\widetilde{S}_i \times \widetilde{U}_i \right]^- = \hphantom{^-} \widetilde{S}_i \times \partial \widetilde{U}_i, \quad \left[\widetilde{Q}_i \times \widetilde{P}_i \right]^- = \hphantom{^-} \widetilde{Q}_i \times \partial \widetilde{P}_i.
\end{equation}
If we choose the sizes of the constituent rectangles to be
\begin{equation}
\left| \widetilde{S}_i \right| = \tilde{\alpha}_i, \quad \left| \widetilde{U}_i \right| = \tilde{\beta}_i, \quad \left| \widetilde{Q}_i \right| = \tilde{\gamma}_i, \quad \left| \widetilde{P}_i \right| = \tilde{\delta}_i
\end{equation}
where
\begin{align}
\tilde{\alpha}_i \hphantom{^-} >& \hphantom{^-} \left(\alpha_i + 2 \nu_i \right) \lambda_+^{N_i}, \label{eq_step1_ineqs1} \\
\tilde{\beta}_i \hphantom{^-} <& \hphantom{^-} \beta_i \mu_-^{N_i}, \label{eq_step1_ineqs2} \\
\tilde{\gamma}_i \hphantom{^-} >& \hphantom{^-} \gamma_i + N_i T_+ \delta_i + C N_i^2 \epsilon^k, \label{eq_step1_ineqs3} \\
\tilde{\delta}_i \hphantom{^-} <& \hphantom{^-} \delta_i - CN_i \epsilon^k, \label{eq_step1_ineqs4}
\end{align}
then $W_i$ is correctly aligned with $\widetilde{W}_i$ under $F^{N_i}$ by Theorem \ref{theorem_cawproduct}. In Section \ref{section_proof1_part3} we explain how to choose $\alpha_i, \beta_i, \gamma_i, \delta_i$ and $N_i$ so that inequalities \eqref{eq_step1_ineqs1}, \eqref{eq_step1_ineqs2}, \eqref{eq_step1_ineqs3}, \eqref{eq_step1_ineqs4} are solvable. 

\begin{remark}\label{remark_whynuhas2}
The reason $\nu_i$ carries a factor of 2 in \eqref{eq_step1_ineqs1} is the following. Since the hyperbolic rectangle of $F^{N_i}\left(W_i \right)$ lies to one side of $\{ s=0 \}$ (see Figure \ref{figure_step1} (a)) we must choose $\tilde{\alpha}_i$ so that $\frac{\tilde{\alpha}_i}{2}$ is greater than the distance from $\{ s=0 \}$ to the outermost point on $F^{N_i}\left(W_i \right)$ in the $s$-direction, which is at most $\lambda_+^{N_i} \left( \frac{\alpha_i}{2} + \nu_i \right)$. 
\end{remark}

\subsubsection{Step 2}

Take a forward iterate $F^{K_i} \left( \widetilde{W}_i \right)$ of $\widetilde{W}_i$ that brings the centre $f^{N_i} \left(y_i^s \right)$ of $\widetilde{W}_i$ to $f^{N_i + K_i} \left(y_i^s \right)$. Define a window $\widehat{W}_i$ centred at $f^{N_i + K_i} \left(y_i^s \right)$, and given in $(s,u,q,p)$ coordinates by
\begin{equation}
\widehat{W}_i = \left[ \widehat{S}_i \times \widehat{U}_i \right] \times \left[ \widehat{Q}_i \times \widehat{P}_i \right]
\end{equation}
where $\widehat{S}_i, \widehat{U}_i, \widehat{Q}_i, \widehat{P}_i$ are rectangles in $s,u,q,p$ respectively, and where we choose the exit sets to be
\begin{equation}
\left[ \widehat{S}_i \times \widehat{U}_i \right]^- = \hphantom{^-} \widehat{S}_i \times \partial \widehat{U}_i, \quad \left[ \widehat{Q}_i \times \widehat{P}_i \right]^- = \hphantom{^-} \partial \widehat{Q}_i \times \widehat{P}_i.
\end{equation}

\begin{figure}[h]
    \centering
    \begin{subfigure}[b]{0.45\textwidth}
        \includegraphics[width=\textwidth]{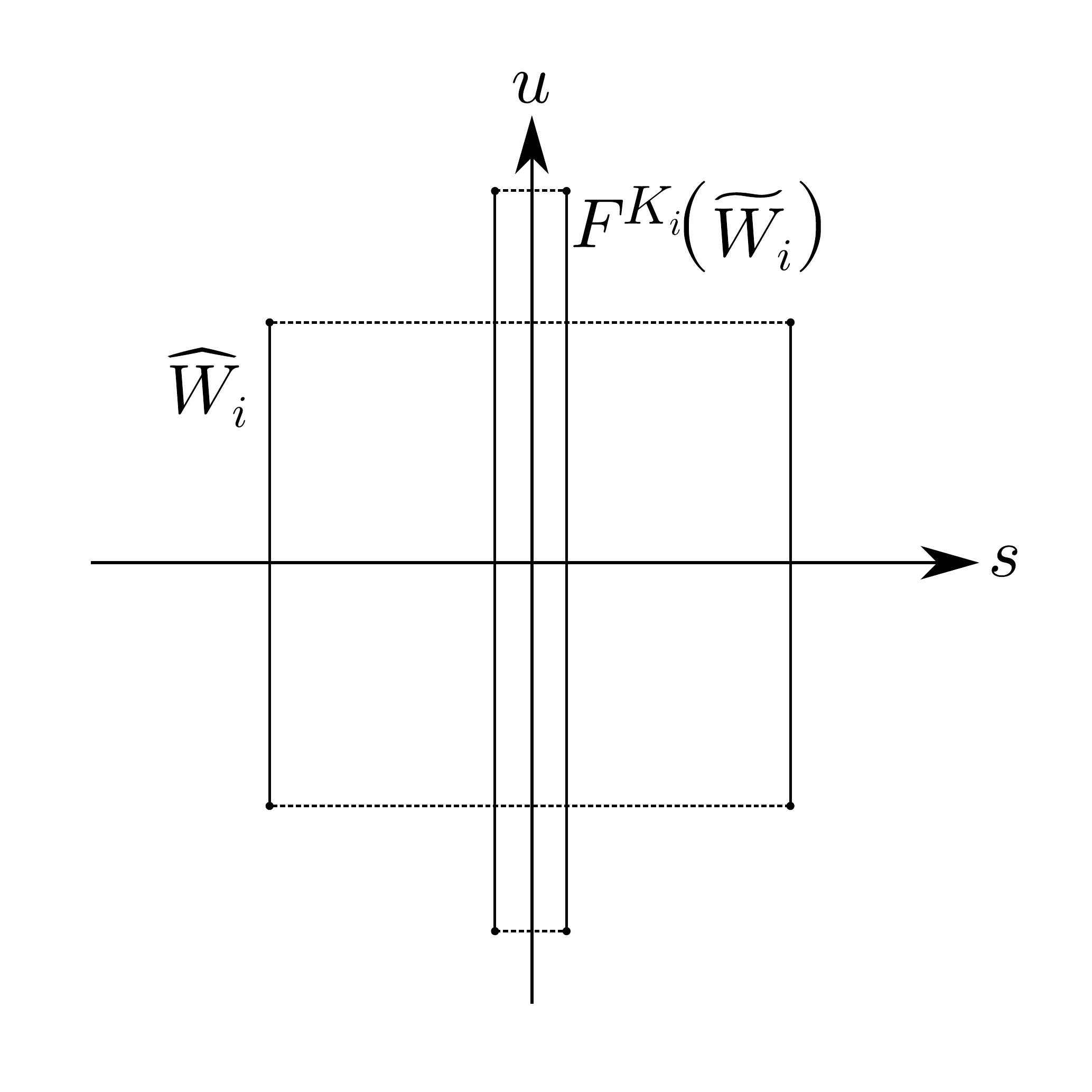}
        \caption{}
        \label{figure_step2}
    \end{subfigure} 
    \begin{subfigure}[b]{0.45\textwidth}
        \includegraphics[width=\textwidth]{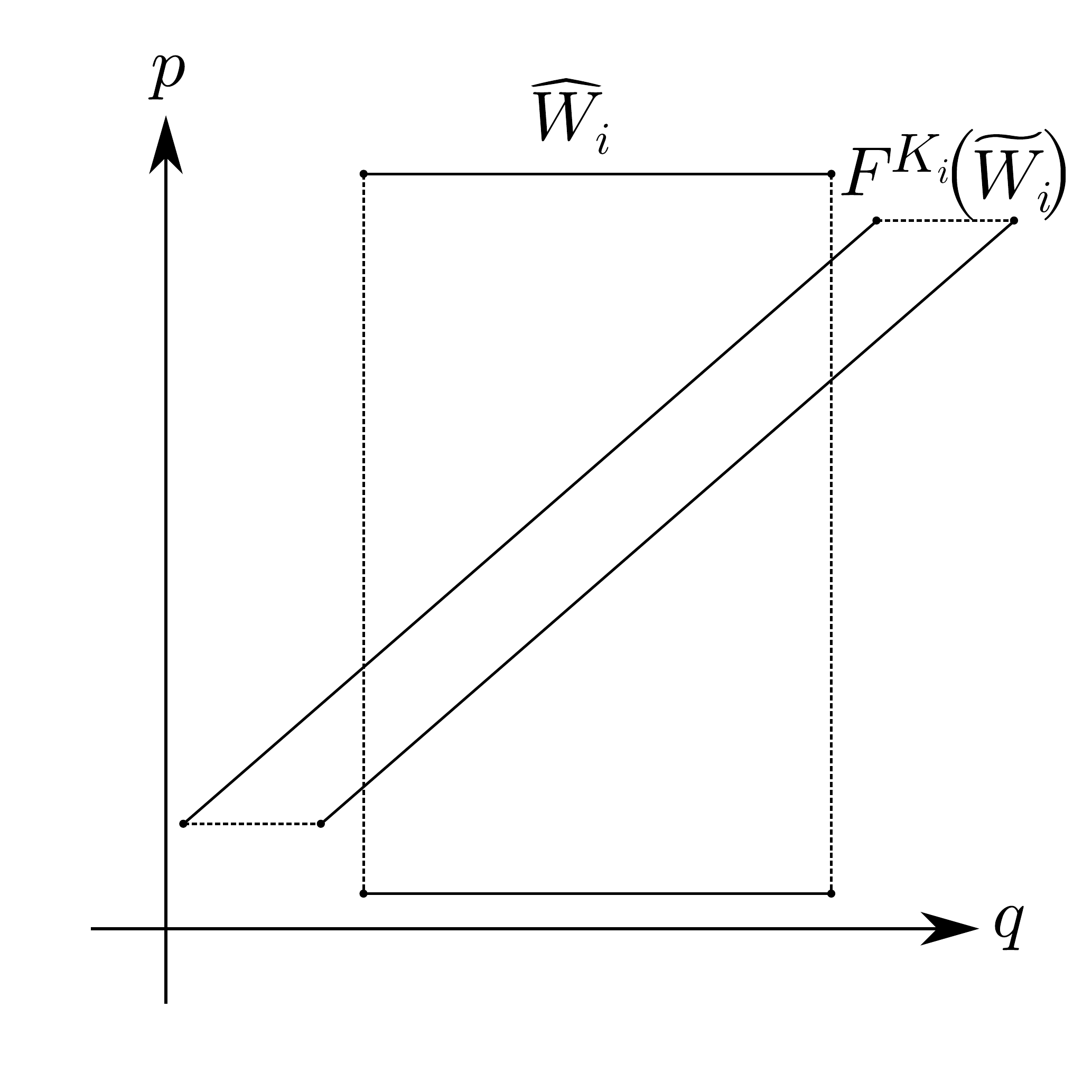}
        \caption{}
        \label{figure_step2}
    \end{subfigure}
    \caption{Step 2. The dashed edges represent exit sets.}\label{figure_step2}
\end{figure}

Notice that the exit set of $\widehat{W}_i$ is not in the same direction as the exit set of $\widetilde{W}_i$. Under the iteration $F^{K_i}$, the window $\widetilde{W}_i$ is deformed, by \eqref{eq_normalhyperbolicitydef} and Lemma \ref{lemma_shearingestimates}, as follows:
\begin{itemize}
\item
In the $s$-direction, $\widetilde{W}_i$ is contracted by $F^{K_i}$ to a size between $\tilde{\alpha}_i \lambda_-^{K_i}$ and $\tilde{\alpha}_i \lambda_+^{K_i}$;
\item
In the $u$-direction, $\widetilde{W}_i$ is expanded by $F^{K_i}$ to a size between $\tilde{\beta}_i \mu_-^{K_i}$ and $\tilde{\beta}_i \mu_+^{K_i}$;
\item
In the $q$-direction, $\widetilde{W}_i$ is sheared by the twist map to a size at least $\epsilon^{\tau} K_i T_- \tilde{\delta}_i - K_i R \tilde{\delta}_i^2 - \tilde{\gamma}_i - C K_i \epsilon^k$; and
\item
In the $p$-direction, the change is $\pm C K_i \epsilon^k$.
\end{itemize}
If we choose the sizes of the constituent rectangles of $\widehat{W}_i$ to be
\begin{equation}
\left| \widehat{S}_i \right| = \hat{\alpha}_i, \quad \left| \widehat{U}_i \right| = \hat{\beta}_i, \quad \left| \widehat{Q}_i \right| = \hat{\gamma}_i, \quad \left| \widehat{P}_i \right| = \hat{\delta}_i
\end{equation}
where
\begin{align}
\hat{\alpha}_i \hphantom{^-} >& \hphantom{^-} \tilde{\alpha}_i \lambda_+^{K_i}, \label{eq_step2_ineqs1} \\
\hat{\beta}_i \hphantom{^-} <& \hphantom{^-} \tilde{\beta}_i \mu_-^{K_i}, \label{eq_step2_ineqs2} \\
\hat{\gamma}_i \hphantom{^-} <& \hphantom{^-} \epsilon^{\tau} K_i T_- \tilde{\delta}_i - K_i R \tilde{\delta}_i^2 - \tilde{\gamma}_i - C K_i^2 \epsilon^k, \label{eq_step2_ineqs3} \\
\hat{\delta}_i \hphantom{^-} >& \hphantom{^-} \tilde{\delta}_i + C K_i \epsilon^k, \label{eq_step2_ineqs4}
\end{align}
then $\widetilde{W}_i$ is correctly aligned with $\widehat{W}_i$ under $F^{K_i}$ by Theorem \ref{theorem_cawproduct}. In Section \ref{section_proof1_part3} we explain how to choose $\tilde{\alpha}_i, \tilde{\beta}_i, \tilde{\gamma}_i, \tilde{\delta}_i$ and $K_i$ so that inequalities \eqref{eq_step2_ineqs1}, \eqref{eq_step2_ineqs2}, \eqref{eq_step2_ineqs3}, \eqref{eq_step2_ineqs4} are solvable.

\subsubsection{Step 3}

Let $x_{i+1} \in \Gamma$ be a point in the homoclinic channel for which there are $y_i^u \in \L_i$ and $y_{i+1}^s \in \L_{i+1}$ such that $\pi^u(x_{i+1}) = y_i^u$ and $\pi^s(x_{i+1}) = y_{i+1}^s$. Define a window $W_i'$ centred at $x_{i+1}$ and given in $(s^-,u^-,q^-,p^-)$ coordinates by 
\begin{equation}
W_i' = \left[ S_i' \times U_i' \right] \times \left[ Q_i' \times P_i' \right]
\end{equation}
where we choose the exit sets to be
\begin{equation}
\left[ S_i' \times U_i' \right]^- = \hphantom{^-} S_i' \times \partial U_i', \quad \left[ Q_i' \times P_i' \right]^- = \hphantom{^-} \partial Q_i' \times P_i'.
\end{equation}
Suppose we choose the sizes of the constituent rectangles of $W_i'$ to be
\begin{equation}
\left| S_i' \right| = \alpha_i', \quad \left| U_i' \right| = \beta_i', \quad \left| Q_i' \right| = \gamma_i', \quad \left| P_i' \right| = \delta_i'.
\end{equation}
We take an iterate $F^{M_i} \left(\widehat{W}_i \right)$ of $\widehat{W}_i$ for some $M_i \in \mathbb{N}$, and require that $F^{M_i} \left(\widehat{W}_i \right)$ is correctly aligned with $W_i'$ under $F^{M_i}$. Observe that this is equivalent to $W_i'$ being correctly aligned with $\widehat{W}_i$ under $F^{-M_i}$, but with the roles of entry and exit sets in the definition of correctly aligned windows reversed. 

\begin{figure}[h]
    \centering
    \begin{subfigure}[b]{0.45\textwidth}
        \includegraphics[width=\textwidth]{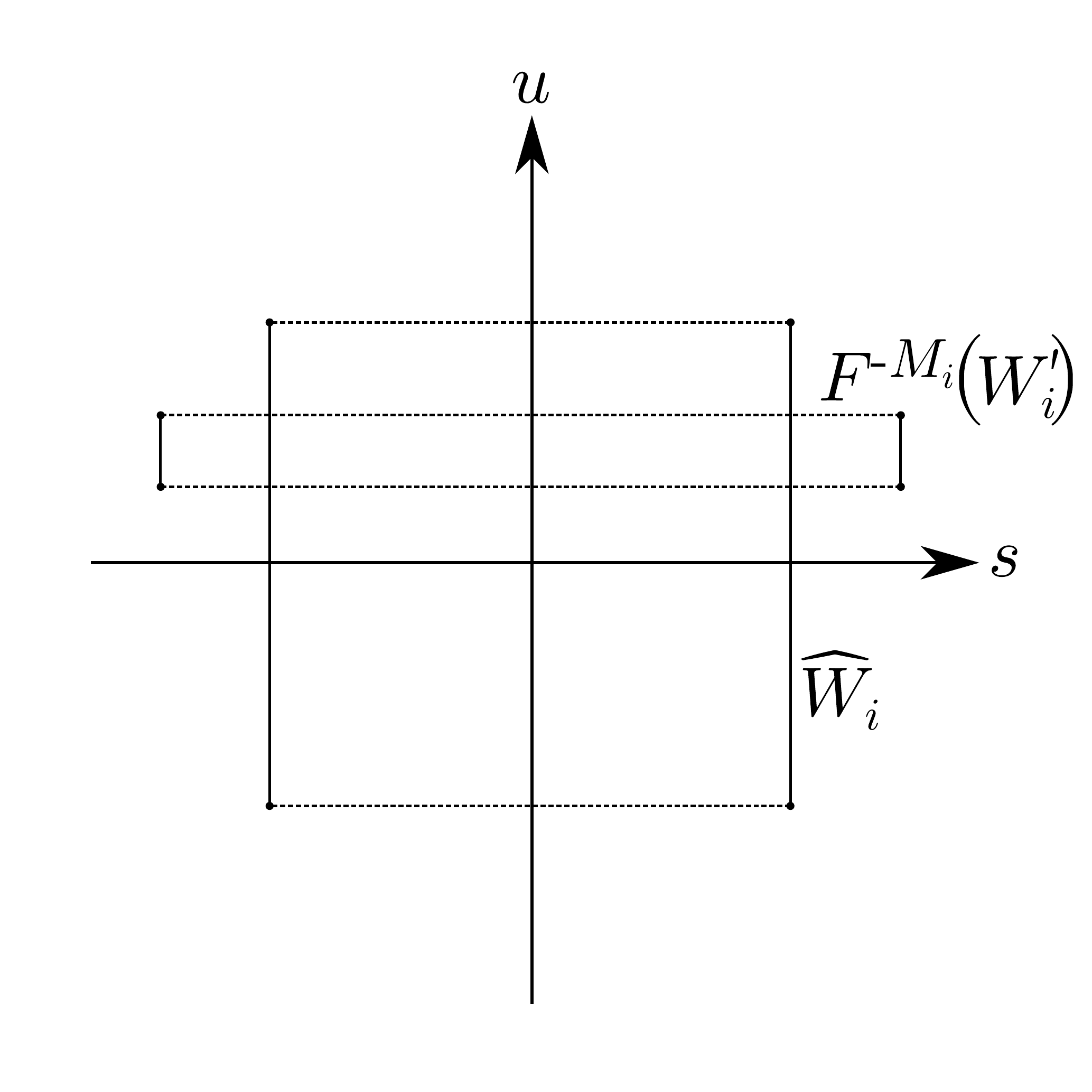}
        \caption{}
        \label{figure_step3}
    \end{subfigure} 
    \begin{subfigure}[b]{0.45\textwidth}
        \includegraphics[width=\textwidth]{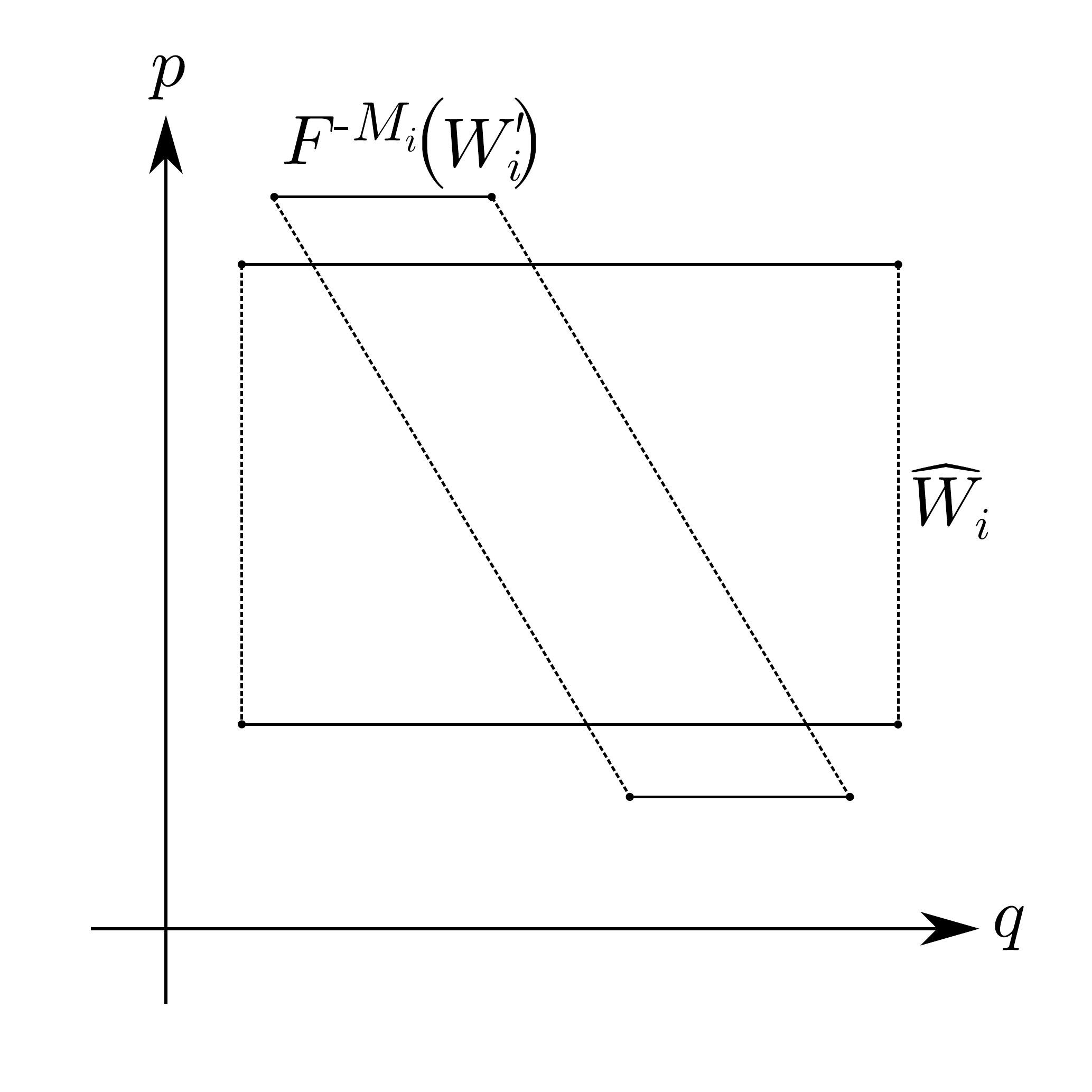}
        \caption{}
        \label{figure_step2}
    \end{subfigure}
    \caption{Step 3. The dashed edges represent exit sets.}\label{figure_step3}
\end{figure}

We can write $y_i^u = (q_i^u,p_i^u)$ and $f^{N_i+K_i}(y_i^s) = (q_i^s,p_i^s)$. Let 
\begin{equation} \label{eq_omegadistancedef}
\omega_i' = \left\| q_i^s - q_i^u \right\|, \quad \nu_i' = d_{W^u(\Lambda)} (x_{i+1},y_i^u)
\end{equation}
where $d_{W^u (\Lambda)}$ is the distance measured along the unstable manifold. Then
\begin{equation}
\nu_i' \mu_+^{-M_i} \leq d_{W^u(\Lambda)} (F^{-M_i}(x_{i+1}), F^{-M_i}(y_i^u)) \leq \nu_i' \mu_-^{-M_i},
\end{equation}
and moreover, by \eqref{eq_normalhyperbolicitydef} and Lemma \ref{lemma_shearingestimates}:
\begin{itemize}
\item
In the $s$-direction, $W_i'$ is expanded by $F^{-M_i}$ to a size between $\alpha_i' \lambda_+^{-M_i}$ and $\alpha_i' \lambda_-^{-M_i}$;
\item
In the $u$-direction, $W_i'$ is contracted by $F^{-M_i}$ to a size between $\beta_i'  \mu_+^{-M_i}$ and $\beta_i'  \mu_-^{-M_i}$;
\item
In the $q$-direction, $W_i'$ is sheared by the twist map to a size at most $\gamma_i' + M_i T_+ \delta_i' + C M_i \epsilon^k$; and
\item
In the $p$-direction, the change is $\pm C M_i \epsilon^k$.
\end{itemize}
If we choose $\alpha_i', \beta_i', \gamma_i', \delta_i'$ such that
\begin{align}
\alpha_i' \lambda_+^{-M_i} \hphantom{^-} >& \hphantom{^-} \hat{\alpha}_i, \label{eq_step3_ineqs1} \\
\left( \beta_i' + 2 \nu_i' \right) \mu_-^{-M_i} \hphantom{^-} <& \hphantom{^-} \hat{\beta}_i, \label{eq_step3_ineqs2} \\
\gamma_i' + M_i T_+ \delta_i' + C M_i^2 \epsilon^k + 2 \omega_i' \hphantom{^-} <& \hphantom{^-} \hat{\gamma}_i, \label{eq_step3_ineqs3} \\
\delta_i' - CM_i \epsilon^k \hphantom{^-} >& \hphantom{^-} \hat{\delta}_i, \label{eq_step3_ineqs4}
\end{align}
then $\widehat{W}_i$ is correctly aligned with $W_i'$ under $F^{M_i}$ by Theorem \ref{theorem_cawproduct}. In Section \ref{section_proof1_part3} we explain how to choose $\hat{\alpha}_i, \hat{\beta}_i, \hat{\gamma}_i, \hat{\delta}_i$ and $M_i$ so that inequalities \eqref{eq_step3_ineqs1}, \eqref{eq_step3_ineqs2}, \eqref{eq_step3_ineqs3}, \eqref{eq_step3_ineqs4} are solvable.

\begin{remark}
With regards to the factor of 2 carried both by $\nu_i'$ in \eqref{eq_step3_ineqs2} and by $\omega_i'$ in \eqref{eq_step3_ineqs3}, compare Remark \ref{remark_whynuhas2} with Figure \ref{figure_step3}.
\end{remark}

\subsection{Continuing the Sequence Across a Transverse Homoclinic Intersection}\label{sec_coordtransfhomoclinic}

So far, we have shown how to construct short sequences of windows $W_i, \widetilde{W}_i, \widehat{W}_i, W_i'$ beginning at the transverse homoclinic point $x_i$, passing near $\L_i$, and ending at another transverse homoclinic point $x_{i+1}$. We now want to show that this short sequence can be connected to the next short sequence $W_{i+1}, \widetilde{W}_{i+1}, \widehat{W}_{i+1}, W_{i+1}'$, which amounts to proving that $W_i'$ is correctly aligned with $W_{i+1}$ under the identity mapping. Both of these windows are centred at $x_{i+1}$, but the difficulty is that $W_i'$ is defined in the $(s^-,u^-,q^-,p^-)$ coordinates obtained by carrying the $(s,u,q,p)$ coordinates from a neighbourhood of $\L_i$ out along $W^u (\Lambda)$ to a neighbourhood of $x_{i+1}$, whereas $W_{i+1}$ is defined in the $(s^+,u^+,q^+,p^+)$ coordinates obtained by carrying the $(s,u,q,p)$ coordinates from a neighbourhood of $\L_{i+1}$ back along $W^s (\Lambda)$ to a neighbourhood of $x_{i+1}$. 

Define the window $W_{i+1}$ centred at $x_{i+1}$ and given in $(s^+ ,u^+,q^+,p^+)$ coordinates by 
\begin{equation}
W_{i+1} = \left[ S_{i+1} \times U_{i+1} \right] \times \left[ Q_{i+1} \times P_{i+1} \right]
\end{equation}
where the exit sets are
\begin{equation}
\left[ S_{i+1} \times U_{i+1} \right]^- =  S_{i+1} \times \partial U_{i+1}, \quad \left[ Q_{i+1} \times P_{i+1} \right]^- = Q_{i+1} \times \partial P_{i+1}
\end{equation}
and where the sizes of the constituent rectangles are
\begin{equation}
\left| S_{i+1} \right| = \alpha_{i+1}, \quad \left| U_{i+1} \right| = \beta_{i+1}, \quad \left| Q_{i+1} \right| = \gamma_{i+1}, \quad \left| P_{i+1} \right| = \delta_{i+1}.
\end{equation}
The following lemma gives conditions under which $W_i'$ is correctly aligned with $W_{i+1}$ under the identity mapping. 

\begin{lemma}\label{lemma_coordtransf}
There are nonnegative constants $C_j$ for $j=1, \ldots, 8$ where $C_4, C_7 >0$ and a constant $R' >0$ such that, if 
\begin{align}
\alpha_{i+1} >& \hphantom{_{i+1}} C_1 \epsilon^{\sigma} \alpha_i' + C_2 \beta_i' + R' \zeta_i^2, \label{eq_coordchange1}\\
\beta_{i+1} < & -C_3 \alpha_i' + C_4 \epsilon^{\sigma} \beta_i' - R' \zeta_i^2, \label{eq_coordchange2}\\
\gamma_{i+1} >& \hphantom{_{i+1}} C_5 \gamma_i' + C_6 \epsilon^{\upsilon} \delta_i' + R' \zeta_i^2, \label{eq_coordchange3}\\
\delta_{i+1} <& \hphantom{_{i+1}} C_7 \epsilon^{\upsilon} \gamma_i' - C_8 \delta_i' - R' \zeta_i^2, \label{eq_coordchange4}
\end{align}
where
\begin{equation}\label{eq_zetaidef}
\zeta_i = \max \left\{ \alpha_i', \beta_i', \gamma_i', \delta_i' \right\}
\end{equation}
then $W_i'$ is correctly aligned with $W_{i+1}$ under the identity mapping. 
\end{lemma}

\begin{proof}
Recall we assume that the stable and unstable manifolds $W^{s,u} \left( \Lambda \right)$ have equal dimension $m+2n$, where we denote by $m$ the dimension of the $s, u$ variables, and by $n$ the dimension of the $q,p$ variables. By [A1], the angle between $s^-$ and $u^+$ at $x_{i+1}$ is of order $\epsilon^{\sigma}$, as is the angle between $u^-$ and $s^+$. Moreover, by [A3], the angle between $q^-$ and $q^+$ at $x_{i+1}$ is of order $\epsilon^{\upsilon}$, as is the angle between $p^-$ and $p^+$. It follows that there is a well-defined coordinate transformation $\phi : (s^-,u^-,q^-,p^-) \mapsto (s^+,u^+,q^+,p^+)$ in a neighbourhood of $x_{i+1}$, and the linearisation of $\phi$ at $x_{i+1}$ is of the form
\begin{equation}\label{eq_coordtransflinearisation}
D \phi (x_{i+1})= \left(
\begin{matrix}
A & 0 \\
0 & B \\
\end{matrix}
\right)
\end{equation}
where 
\begin{equation}\label{eq_coordtransflinearisationcomps}
A = \left(
\begin{matrix}
\epsilon^{\sigma} A_1  & A_2 \\
A_3 & \epsilon^{\sigma} A_4  \\
\end{matrix}
\right), \quad B = \left(
\begin{matrix}
B_1 & \epsilon^{\upsilon} B_2  \\
\epsilon^{\upsilon} B_3  & B_4 \\
\end{matrix}
\right).
\end{equation} 
Here each matrix $A_j$ is of dimension $m \times m$, each matrix $B_j$ is of dimension $n \times n$, and $A_1, A_4, B_2, B_3$ are invertible.

Recall the maximum norm $\| . \|$ is defined by \eqref{eq_maxnormdef}. Fix $R' >0$ with the following property. We can write
\begin{equation}\label{eq_coordtransftaylor}
\phi(x) = \phi(x_{i+1}) + D \phi (x_{i+1}) + R_*' (x)
\end{equation}
where $R_*'(x)$ is a remainder term of order $O \left( \left\|x - x_{i+1}  \right\|^2 \right)$. Then the constant $R'$ is chosen so that
\begin{equation}\label{eq_coordtransfremainderbound}
\left\| R_*'(x) \right\| \leq R' \left\| x - x_{i+1} \right\|^2
\end{equation}
for all $x$ in some neighbourhood of $x_{i+1}$. Notice moreover that, since $W_i'$ is centred at $x_{i+1}$, we have
\begin{equation}\label{eq_zetaiinequality}
\left\| x - x_{i+1} \right\| \leq \frac{\zeta_i}{2}
\end{equation}
for all $x \in W_i'$, where $\zeta_i$ is defined by \eqref{eq_zetaidef}. 

There are $a_i', b_i' \in \mathbb{R}^m$ and $c_i', d_i' \in \mathbb{R}^n$ such that
\begin{align}
S_i' ={}& \left[ a_{i,1}', a_{i,1}' + \alpha_i' \right] \times \cdots \times \left[ a_{i,m}', a_{i,m}' + \alpha_i' \right], \\
U_i' ={}& \left[ b_{i,1}', b_{i,1}' + \beta_i' \right] \times \cdots \times \left[ b_{i,m}', b_{i,m}' + \beta_i' \right], \\
Q_i' ={}& \left[ c_{i,1}', c_{i,1}' + \gamma_i' \right] \times \cdots \times \left[ c_{i,n}', c_{i,n}' + \gamma_i' \right], \\
P_i' ={}& \left[ d_{i,1}', d_{i,1}' + \delta_i' \right] \times \cdots \times \left[ d_{i,n}', d_{i,n}' + \delta_i' \right].
\end{align}
For each $j=1, \ldots, m$ define
\begin{align}
I_{j,0}^h ={}& \left[ b_{i,1}', b_{i,1}' + \beta_i' \right] \times \cdots \times \left\{ b_{i,j}' \right\} \times \cdots \times \left[ b_{i,m}', b_{i,m}' + \beta_i' \right], \\
I_{j,1}^h ={}& \left[ b_{i,1}', b_{i,1}' + \beta_i' \right] \times \cdots \times \left\{ b_{i,j}' + \beta_i' \right\} \times \cdots \times \left[ b_{i,m}', b_{i,m}' + \beta_i' \right],
\end{align}
and for each $k=1, \ldots, n$ define
\begin{align}
I_{k,0}^c ={}& \left[ c_{i,1}', c_{i,1}' + \gamma_i' \right] \times \cdots \times \left\{ c_{i,k}' \right\} \times \cdots \times \left[ c_{i,n}', c_{i,n}' + \gamma_i' \right], \\
I_{k,1}^c ={}& \left[ c_{i,1}', c_{i,1}' + \gamma_i' \right] \times \cdots \times \left\{ c_{i,k}' + \gamma_i' \right\} \times \cdots \times \left[ c_{i,n}', c_{i,n}' + \gamma_i' \right].
\end{align}
We then let
\begin{equation}
E_{j,0}^h= S_i' \times I^h_{j,0} \times Q_i' \times P_i', \quad E_{j,1}^h= S_i' \times I^h_{j,1} \times Q_i' \times P_i'
\end{equation}
be the corresponding components of the exit set of $W_i'$ in the hyperbolic directions, and let
\begin{equation}
E_{k,0}^c= S_i' \times U_i' \times I^c_{k,0} \times P_i', \quad E_{k,1}^c= S_i' \times U_i' \times I^c_{k,1} \times P_i'
\end{equation}
be the corresponding components of the exit set of $W_i'$ in the cylindrical directions. For each $\kappa \in \left\{ s,u,q,p \right\}$ denote by $\Pi_{\kappa} : (s,u,q,p) \mapsto \kappa$ the projection onto the $\kappa$ coordinate. Define
\begin{align}
\begin{split}
\Delta^h_j ={}& \min_{z_0 \in E_{j,0}^h} \min_{z_1 \in E_{j,1}^h} \left\| \Pi_u \circ \phi \left( z_1 \right) - \Pi_u \circ \phi \left( z_0 \right) \right\|, \label{eq_coordtransfdeltajh} 
\end{split} \\
\begin{split}
\Omega^h_j ={}& \max_{z_0 \in E_{j,0}^h} \max_{z_1 \in E_{j,1}^h} \left\| \Pi_s \circ \phi \left( z_1 \right) - \Pi_s \circ \phi \left( z_0 \right) \right\|, \label{eq_coordtransfomegajh} 
\end{split} \\
\begin{split}
\Delta^c_k ={}& \min_{z_0 \in E_{k,0}^c} \min_{z_1 \in E_{k,1}^c} \left\| \Pi_p \circ \phi \left( z_1 \right) - \Pi_p \circ \phi \left( z_0 \right) \right\|, \label{eq_coordtransfdeltakc} 
\end{split} \\
\begin{split}
\Omega^c_k ={}& \max_{z_0 \in E_{k,0}^c} \max_{z_1 \in E_{k,1}^c} \left\| \Pi_q \circ \phi \left( z_1 \right) - \Pi_q \circ \phi \left( z_0 \right) \right\|. \label{eq_coordtransfomegack}  
\end{split}
\end{align}
If we choose 
\begin{equation} \label{eq_windowiplus1aspectratios}
\alpha_{i+1} > \Omega_j^h, \quad \beta_{i+1} < \Delta_j^h, \quad \gamma_{i+1} > \Omega_k^c, \quad \delta_{i+1} < \Delta_k^c
\end{equation}
for each $j=1, \ldots, m$ and each $k = 1, \ldots, n$ then $W_i'$ is correctly aligned with $W_{i+1}$ under the map $\phi$ (see Figure \ref{figure_coordchange}). Therefore, in what follows, we search for upper bounds for each $\Omega^h_j, \Omega^c_k$ and lower bounds for each $\Delta^h_j, \Delta^c_k$. In fact, we compute bounds for $\Omega^h_1, \Omega^c_1$ and $\Delta^h_1, \Delta^c_1$ as the estimates for other values of $j,k$ are analogous. 

\begin{figure}
    \centering
    \begin{subfigure}[b]{0.45\textwidth}
        \includegraphics[width=\textwidth]{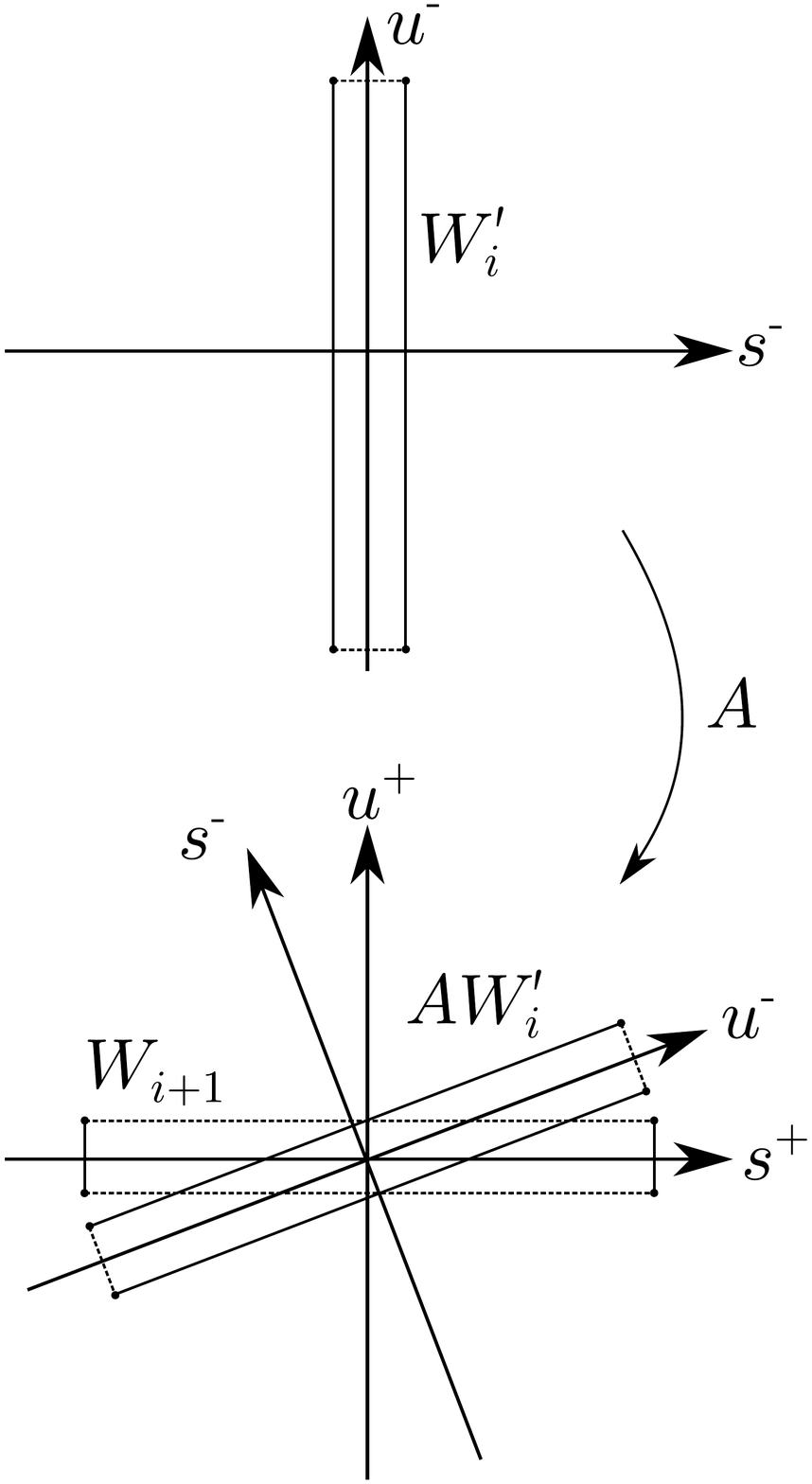}
        \caption{}
        \label{figure_coordchange}
    \end{subfigure} 
    \begin{subfigure}[b]{0.45\textwidth}
        \includegraphics[width=\textwidth]{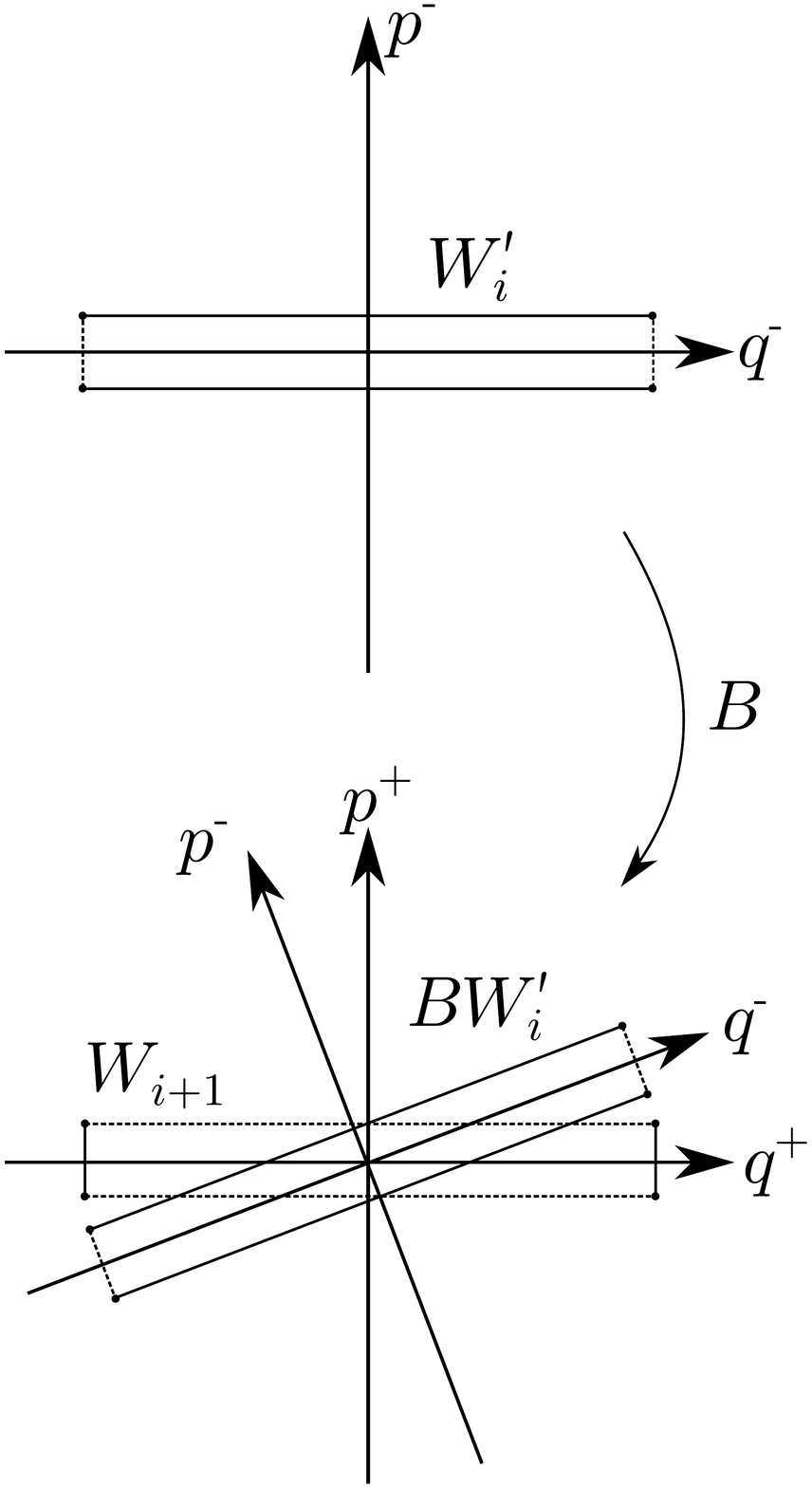}
        \caption{}
        \label{figure_coordchange}
    \end{subfigure}
    \caption{The figure illustrates how the window $W_i'$ is transformed under the linearisation $D \phi \left( x_{i+1} \right)$ of the change of coordinates $\phi : (s^-,u^-,q^-,p^-) \mapsto (s^+,u^+,q^+,p^+)$ at the homoclinic point $x_{i+1}$. In (a), we see how the hyperbolic rectangle is transformed assuming [A1], i.e. that there is an angle of order $\epsilon^{\sigma}$ between the stable and unstable manifolds at the homoclinic point $x_{i+1}$. In (b), we see how the inner (cylindrical) rectangle is transformed assuming [A3], i.e. that the angle between the image of a leaf $\L_i$ of the foliation under the scattering map and another leaf $\L_{i+1}$ is of order $\epsilon^{\upsilon}$. Note that the rectangles are not centred at the origin. The dashed edges of the rectangles represent the exit sets.}\label{figure_coordchange}
\end{figure}

For $j=0,1$ let $s_j^h \in S_i'$, $q_j^h \in Q_i'$, $p_j^h \in P_i'$, and
\begin{equation}
u_j^* \in \left[ b_{i,2}', b_{i,2}' + \beta_i' \right] \times \cdots \times \left[ b_{i,m}', b_{i,m}' + \beta_i' \right].
\end{equation} 
Define $u^h_0 = \left( b_{i,1}', u_0^* \right)$, $u^h_1 = \left( b_{i,1}' + \beta_i', u_1^* \right)$, and $z_j^h = \left( s_j^h, u_j^h, q_j^h, p_j^h \right)$ for $j=0,1$. Observe that all points in $E_{1,j}^h$ are of the form $z_j^h$. Moreover, notice that
\begin{equation}\label{eq_hypucompdiff}
\left\| u_1^h - u_0^h \right\| = \beta_i'. 
\end{equation}
Due to \eqref{eq_coordtransflinearisation} and \eqref{eq_coordtransflinearisationcomps}, the hyperbolic part of $D \phi \left( x_{i+1} \right) \left( z_1^h - z_0^h \right)$ is
\begin{equation}\label{eq_coordtransfhyperbolicpart}
\left(
\begin{matrix}
\epsilon^{\sigma} A_1  & A_2 \\
A_3 & \epsilon^{\sigma} A_4  
\end{matrix}
\right) \left(
\begin{matrix}
s_1^h - s_0^h \\
u_1^h - u_0^h
\end{matrix}
\right) = \left(
\begin{matrix}
\epsilon^{\sigma} A_1 \left(s_1^h - s_0^h \right) + A_2 \left( u_1^h - u_0^h \right) \\
A_3 \left(s_1^h - s_0^h \right) + \epsilon^{\sigma} A_4 \left( u_1^h - u_0^h \right)
\end{matrix}
\right). 
\end{equation}
It follows from \eqref{eq_coordtransftaylor}, \eqref{eq_coordtransfremainderbound}, \eqref{eq_zetaiinequality}, \eqref{eq_hypucompdiff}, and \eqref{eq_coordtransfhyperbolicpart} that
\begin{align}
\left\| \Pi_s \circ \phi \left( z_1^h \right) - \Pi_s \circ \phi \left( z_0^h \right) \right\| \leq{}& \left\| \Pi_s \left( D \phi \left( x_{i+1} \right) \left( z_1^h - z_0^h \right) \right) \right\| \\
& \qquad + R' \left( \left\| z_1^h - x_{i+1} \right\|^2 + \left\| z_0^h - x_{i+1} \right\|^2 \right) \\
\leq& \left\| \epsilon^{\sigma} A_1 \left(s_1^h - s_0^h \right) + A_2 \left( u_1^h - u_0^h \right) \right\| + \frac{1}{2} R' \zeta_i^2 \\
\leq& C_1 \epsilon^{\sigma} \alpha_i' + C_2 \beta_i' + R' \zeta_i^2
\end{align}
and
\begin{align}
\left\| \Pi_u \circ \phi \left( z_1^h \right) - \Pi_u \circ \phi \left( z_0^h \right) \right\| \geq{}& \left\| \Pi_u \left( D \phi \left( x_{i+1} \right) \left( z_1^h - z_0^h \right) \right) \right\| \\
& \qquad - R' \left( \left\| z_1^h - x_{i+1} \right\|^2 + \left\| z_0^h - x_{i+1} \right\|^2 \right) \\
\geq{}& \left\| A_3 \left(s_1^h - s_0^h \right) + \epsilon^{\sigma} A_4 \left( u_1^h - u_0^h \right) \right\| - \frac{1}{2} R' \zeta_i^2 \\
\geq{}& - \left\| A_3 \left(s_1^h - s_0^h \right) \right\| + \epsilon^{\sigma} \left\| A_4 \left(u_1^h - u_0^h \right) \right\| - R' \zeta_i^2 \\
\geq{}& - C_3 \alpha_i' + C_4 \epsilon^{\sigma} \beta_i' - R' \zeta_i^2
\end{align}
where $\zeta_i$ is defined by \eqref{eq_zetaidef}, where $C_j$ is the matrix norm of $A_j$ for $j=1,2,3$, and where $C_4>0$ is such that $\left\| A_4 v \right\| \geq C_4 \left\| v \right\|$ for all $v \in \mathbb{R}^m$. That we can choose such a strictly positive $C_4$ follows from the invertibility of $A_4$. Combining the previous inequalities with \eqref{eq_coordtransfdeltajh}, \eqref{eq_coordtransfomegajh}, and \eqref{eq_windowiplus1aspectratios} yields \eqref{eq_coordchange1} and \eqref{eq_coordchange2}. 

Now, for $j=0,1$, let $s_j^c \in S_i'$, $u_j^c \in U_i'$, $p_j^c \in P_i'$, and
\begin{equation}
q_j^* \in \left[ c_{i,2}', c_{i,2}' + \gamma_i' \right] \times \cdots \times \left[ c_{i,n}', c_{i,n}' + \gamma_i' \right].
\end{equation}
Define $q_0^c = \left( c_{i,1}', q_0^* \right)$, $q_1^c = \left( c_{i,1}' + \gamma_i', q_1^* \right)$, and $z_j^c = \left( s_j^c, u_j^c, q_j^c, p_j^c \right)$ for $j=0,1$. Then all points in $E_{1,j}^c$ are of the form $z_j^c$, and furthermore
\begin{equation}\label{eq_cylqcompdiff}
\left\| q_1^c - q_0^c \right\| = \gamma_i'
\end{equation}
Due to \eqref{eq_coordtransflinearisation} and \eqref{eq_coordtransflinearisationcomps}, the cylindrical part of $D \phi \left( x_{i+1} \right) \left( z_1^c - z_0^c \right)$ is
\begin{equation}\label{eq_coordtransfcylindricalpart}
\left(
\begin{matrix}
B_1 & \epsilon^{\upsilon} B_2  \\
\epsilon^{\upsilon} B_3  & B_4 \\
\end{matrix}
\right) \left(
\begin{matrix}
q_1^c - q_0^c \\
p_1^c - p_0^c \\
\end{matrix}
\right) = \left(
\begin{matrix}
B_1 \left( q_1^c - q_0^c \right) + \epsilon^{\upsilon} B_2 \left( p_1^c - p_0^c \right) \\
\epsilon^{\upsilon} B_3 \left( q_1^c - q_0^c \right) + B_4 \left( p_1^c - p_0^c \right)\\
\end{matrix}
\right).
\end{equation}
It follows from \eqref{eq_coordtransftaylor}, \eqref{eq_coordtransfremainderbound}, \eqref{eq_zetaiinequality}, \eqref{eq_cylqcompdiff}, and \eqref{eq_coordtransfcylindricalpart} that
\begin{align}
\left\| \Pi_q \circ \phi \left( z_1^c \right) - \Pi_q \circ \phi \left( z_0^c \right) \right\| \leq{}& \left\| \Pi_q \left( D \phi \left( x_{i+1} \right) \left( z_1^c - z_0^c \right) \right) \right\| \\
& \qquad + R' \left( \left\| z_1^c - x_{i+1} \right\|^2 + \left\| z_0^c - x_{i+1} \right\|^2 \right) \\
\leq& \left\|  B_1 \left(q_1^c - q_0^c \right) + \epsilon^{\upsilon} B_2 \left( p_1^c - p_0^c \right) \right\| + \frac{1}{2} R' \zeta_i^2 \\
\leq& C_5 \gamma_i' + C_6 \epsilon^{\upsilon} \delta_i' + R' \zeta_i^2
\end{align}
and
\begin{align}
\left\| \Pi_p \circ \phi \left( z_1^c \right) - \Pi_p \circ \phi \left( z_0^c \right) \right\| \geq{}& \left\| \Pi_p \left( D \phi \left( x_{i+1} \right) \left( z_1^c - z_0^c \right) \right) \right\| \\
& \qquad - R' \left( \left\| z_1^c - x_{i+1} \right\|^2 + \left\| z_0^c - x_{i+1} \right\|^2 \right) \\
\geq{}& \left\| \epsilon^{\upsilon} B_3 \left(q_1^c - q_0^c \right) + B_4 \left( p_1^c - p_0^c \right) \right\| - \frac{1}{2} R' \zeta_i^2 \\
\geq{}& \epsilon^{\upsilon}  \left\| B_3 \left(q_1^c - q_0^c \right) \right\| - \left\| B_4 \left(p_1^c - p_0^c \right) \right\| - R' \zeta_i^2 \\
\geq{}& C_7 \epsilon^{\upsilon} \gamma_i' - C_8 \delta_i' - R' \zeta_i^2
\end{align}
where $\zeta_i$ is defined by \eqref{eq_zetaidef}, where $C_j$ is the matrix norm of $B_j$ for $j=5,6,8$, and where $C_7>0$ is such that $\left\| B_3 v \right\| \geq C_7 \left\| v \right\|$ for all $v \in \mathbb{R}^n$. As above, the reason that we can choose $C_7$ to be strictly positive is due to the invertibility of $B_3$. The previous inequalities combined with \eqref{eq_coordtransfdeltakc}, \eqref{eq_coordtransfomegack}, and \eqref{eq_windowiplus1aspectratios} give \eqref{eq_coordchange3} and \eqref{eq_coordchange4}. 
\end{proof}

\subsection{Construction of Long Sequences of Correctly Aligned Windows}\label{section_proof1_part3}

In this section we show how to choose the aspect ratios of the windows at each step (see Table \ref{table_aspectratiosorder} for a summary) so that the process can be continued indefinitely, completing the proof of Theorem \ref{theorem_main1}.

Fix $\eta > 0$ as in the statement of Theorem \ref{theorem_main1}. Let us first show how to choose $\alpha_i', \beta_i', \gamma_i', \delta_i'$ so that the inequalities of Lemma \ref{lemma_coordtransf} are solvable; in particular, we require the right-hand side of inequalities \eqref{eq_coordchange2} and \eqref{eq_coordchange4} to be positive. Define $\kappa = \max \left\{ \sigma, \upsilon \right\}$ and suppose we choose
\begin{equation}
\alpha_i' = \epsilon^{2 \kappa} \alpha_*, \quad \beta_i' = \gamma_i' = \epsilon^{\kappa} \zeta_*, \quad  \delta_i' = \epsilon^{2 \kappa} \delta_*
\end{equation}
where
\begin{equation}
\begin{dcases}
0 < \zeta_* < \frac{1}{R'} \min \left\{ C_4, C_7 \right\}, \\
0 < \alpha_* < \min \left\{ \zeta_*, \frac{C_4}{C_3} \zeta_* \left( 1 - \frac{R'}{C_4} \zeta_* \right) \right\}, \\
0 < \delta_* < \min \left\{ \zeta_*, \frac{C_7}{C_8} \zeta_* \left( 1 - \frac{R'}{C_7} \zeta_* \right) \right\}.
\end{dcases}
\end{equation}
It follows that $\beta_i' = \gamma_i' = \zeta_i$, where $\zeta_i$ is defined by \eqref{eq_zetaidef}, and so the right-hand side of \eqref{eq_coordchange2} is 
\begin{equation}
- C_3 \alpha_i' + C_4 \epsilon^{\sigma} \beta_i' - R' \left( \beta_i' \right)^2 \geq \epsilon^{2 \kappa} \left( C_4 \beta_* - C_3 \alpha_* - R' \beta_*^2 \right) > 0.
\end{equation}
Similarly, the right-hand side of \eqref{eq_coordchange4} is
\begin{equation}
C_7 \epsilon^{\upsilon} \gamma_i' - C_8 \delta_i' - R' \left( \gamma_i' \right)^2 \geq \epsilon^{2 \kappa} \left( C_7 \gamma_* - C_8 \delta_* - R' \gamma_*^2 \right) > 0.
\end{equation}
Therefore we can choose $\beta_{i+1}, \delta_{i+1} >0$ of order $O \left( \epsilon^{2 \kappa} \right)$, and $\alpha_{i+1}, \gamma_{i+1} > 0$ of order $O \left(\epsilon^{\kappa}\right)$ (or larger, but not smaller) so that inequalities \eqref{eq_coordchange1}, \eqref{eq_coordchange2}, \eqref{eq_coordchange3}, \eqref{eq_coordchange4} are satisfied. Moreover, we choose
\begin{equation}
\alpha_{i+1}, \beta_{i+1}, \delta_{i+1} < \eta. 
\end{equation}
Notice that $\beta_{i+1}, \delta_{i+1}$ can be chosen as small as we like. 

\begin{table}
\begin{center}
\begin{tabular}{|c|c|c|c|c|c|}
\hline
& ``  '' & $\tilde{}$ & $\hat{}$ & $'$ \\
\hline
$\alpha$ & 1 & 1 & $ \epsilon^{2 \kappa}$ & $\epsilon^{2 \kappa}$ \\[0.5em]
$\beta$ & $ \epsilon^{2 \kappa}$ & $\epsilon^{2 \kappa}$ & 1 & $\epsilon^{\kappa}$ \\[0.5em]
$\gamma$ & $\epsilon^{\kappa}$ & $\epsilon^{\kappa}$ & 1 & $\epsilon^{\kappa}$ \\[0.5em]
$\delta$ & $\epsilon^{2 \kappa}$ & $\epsilon^{\rho}$ & $\epsilon^{2 \kappa}$ & $\epsilon^{2 \kappa}$ \\[0.5em]
\hline 
\end{tabular}
\caption{\label{table_aspectratiosorder}The order of the size of each rectangle at each step in the proof. In addition, we choose the iterates $N_i, M_i = O(1)$ and $K_i = O(\epsilon^{-\rho - \tau})$ to be sufficiently large. }
\end{center}
\end{table}

The next major constraint when choosing the aspect ratios is that $\hat{\gamma}_i$ must be chosen in \eqref{eq_step2_ineqs3} to be of order $O(1)$, and we must be able to choose it as close to 1 as we like, for the following reason. By assumption, there is a point $z \in \L_i$ such that $S (z) \in \L_{i+1}$, and $S \left( U \cap \L_i \right)$ is transverse to $\L_{i+1}$ at $S(z)$, where $S: U \subseteq \Lambda \to \Lambda$ is a branch of the scattering map. It is essential that $z \in \widehat{W}_i$, and so we must be able to choose $\widehat{W}_i$ as wide as necessary in the $q$-direction. This requirement can be seen mathematically in \eqref{eq_step3_ineqs3}: if $\hat{\gamma}_i$ does not dominate $M_i T_+ \delta_i' + C M_i^2 \epsilon^k + 2 \omega_i'$, then $\gamma_i'$ cannot be chosen to be positive. 

Suppose we have $\tilde{\delta}_i = \epsilon^{\rho} \delta_*$ where $\rho, \delta_*$ are to be determined. It follows from \eqref{eq_step1_ineqs4} that $\tilde{\delta}_i = O \left( \epsilon^{2 \kappa} \right)$, and so
\begin{equation}\label{eq_rhogequpsilon}
\rho \geq 2 \kappa.
\end{equation}
Moreover, the first two terms on the right-hand side of \eqref{eq_step2_ineqs3} are
\begin{equation} \label{eq_largetwistexplanation}
K_i \epsilon^{\rho + \tau} \delta_* \left( T_- - R \delta_* \epsilon^{\rho - \tau} \right).
\end{equation}
Since we require that this is positive, we must have 
\begin{equation}\label{eq_rhogeqtau}
\rho \geq \tau
\end{equation}
and $\delta_* < \frac{T_-}{R}$. It follows from \eqref{eq_rhogequpsilon} and \eqref{eq_rhogeqtau} that the value $\rho = \max \{2 \sigma, 2 \upsilon, \tau \}$ defined in \eqref{eq_rhodef} suffices.

Now, for the right-hand side of \eqref{eq_step2_ineqs3} to be positive we require that \eqref{eq_largetwistexplanation} dominates $\tilde{\gamma}_i + C K_i^2 \epsilon^k$, which is true whenever $K_i = O \left(\epsilon^{-\rho - \tau} \right)$ is sufficiently large, where we have used \eqref{eq_kavgcondition}. Due to \eqref{eq_kavgcondition}, the inequality \eqref{eq_step2_ineqs4} is also solvable for $\hat{\delta}_i = O(\epsilon^{2 \kappa})$. Therefore we can again choose $\delta_i'>0$ in \eqref{eq_step3_ineqs4} to be of order $O\left(\epsilon^{2 \kappa}\right)$. Since we can choose $\tilde{\delta}_i$ to be as small as we like (due to \eqref{eq_step1_ineqs4}), we can subsequently choose each $\hat{\delta}_i, \delta_i'$ in \eqref{eq_step2_ineqs4}, \eqref{eq_step3_ineqs4} respectively to be as small as required. 

If we choose $M_i = O(1)$, then \eqref{eq_step3_ineqs4} is solvable. Let us explain how to find $\hat{\gamma}_i$ satisfying \eqref{eq_step3_ineqs3}. If we define $\mathbb{T} = \mathbb{R} / \mathbb{Z}$, then $\omega_i' \leq \frac{1}{2}$, where $\omega_i'$ is defined in \eqref{eq_omegadistancedef}. We may assume, by slightly shifting $y_i^u$ if necessary (to some point $\tilde{y}_i^u$ such that $S \left(\tilde{y}_i^u \right)$ lies in a sufficiently small $O(\eta)$-neighbourhood of $\L_{i+1}$), that $\omega_i' <\frac{1}{2}$. Notice that, by increasing $K_i$ and shrinking $\tilde{\gamma}_i$ if necessary, we can choose $\hat{\gamma}_i$ to be as close to 1 as required. In this way we can find some positive $\gamma_i' = O(1)$ so that the left-hand side of \eqref{eq_step3_ineqs3} is less than 1. Therefore \eqref{eq_step3_ineqs3} is solvable, and we can choose $\gamma_i' = O\left(\epsilon^{\kappa}\right)$.

Suppose we choose $N_i = O(1)$. Then, by \eqref{eq_step1_ineqs1}, $\tilde{\alpha}_i$ is of order 1. Since $K_i = O \left( \epsilon^{- \rho - \tau} \right)$ and since $\lambda_+ \in (0,1)$, we can choose $\hat{\alpha}_i$ to be of order $\epsilon^{2 \kappa}$ in \eqref{eq_step2_ineqs1}. Since $M_i = O (1)$, we can in turn choose $\alpha_i'$ to be of order $\epsilon^{2 \kappa}$ due to \eqref{eq_step3_ineqs1}. Moreover, shrinking $\hat{\alpha}_i$ in \eqref{eq_step2_ineqs1} allows us to shrink $\alpha_i'$ in \eqref{eq_step3_ineqs1} if necessary.
 
Now, since $\beta_i$ is of order $O\left( \epsilon^{2 \kappa}\right)$, so too is $\tilde{\beta}_i$ due to \eqref{eq_step1_ineqs2}. Due to \eqref{eq_step2_ineqs2}, since $K_i = O \left( \epsilon^{- \rho - \tau} \right)$ and since $\mu_- >1$, we can choose $\hat{\beta}_i$ of order 1. Finally, choosing $M_i = O(1)$ large enough, we can ensure that $\hat{\beta}_i - 2 \nu_i' \mu_-^{-M_i} >0$ so that \eqref{eq_step3_ineqs2} is solvable by a positive choice of $\beta_i' = O \left( \epsilon^{\kappa} \right)$.  

We have thus shown that these choices can be made consistently. Combining this construction with Theorem \ref{theorem_onecansee} implies the existence of a trajectory $\{ z_i \}$ as in Theorem \ref{theorem_main1}. Since we choose $M_i, N_i$ to be of order $O(1)$, $K_i$ of order $O \left(\epsilon^{- \rho - \tau} \right)$, the time taken to move from a neighbourhood of $x_i$ to a neighbourhood of $x_{i+1}$ is
\begin{equation}
N_i + K_i + M_i = O \left(\epsilon^{- \rho - \tau} \right).
\end{equation}
In order to move a distance of order $1$ in the $p$-direction, we must choose $N$ consecutive leaves of the foliation connected by the scattering map where $N=O(\epsilon^{-\upsilon})$, and so the time is of order $\epsilon^{-\rho - \tau - \upsilon}$. This concludes the proof of Theorem \ref{theorem_main1}.

\section{Proof of Theorem \ref{theorem_main2}}\label{sec_proofthm2}

Let the notation be as in the statement of Theorem \ref{theorem_main2}, and fix $\eta > 0$. Denote by $\widehat{\Sigma} = \mathbb{R}^{\ell_1} \times [0,1]^{\ell_2}$ the universal cover of $\Sigma$. Suppose we have lifted the dynamics to the covering space $\widehat{M} = M \times \widehat{\Sigma}$ of $\widetilde{M}$. For convenience, we do not change the notation of the lifted mappings. Fix $p^*_1, \ldots, p^*_N \in [0,1]^n$ and $\xi^*_1 \in \mathrm{Int} \left( [0,1]^{\ell_2} \right)$ as in the statement of Theorem \ref{theorem_main2}. Let $F= \widetilde{G} (\cdot; \xi^*_1) \in \Diff^4 (M)$. By [B1], $F$ satisfies the assumptions of Theorem \ref{theorem_main1}. Therefore, by the proof of Theorem \ref{theorem_main1}, there are windows $\widetilde{W}_1, \ldots, \widetilde{W}_N \subset M$ and $n_j \in \mathbb{N}$ such that 
\begin{equation}
d \left( z,  \L_j \right) < \frac{\eta}{2}
\end{equation}
for all $z \in \widetilde{W}_j$, and $\widetilde{W}_j$ is correctly aligned with $\widetilde{W}_{j+1}$ under $F^{n_j}$. For convenience we drop the tilde notation and write simply $W_j$. By Theorem \ref{theorem_cawstable}, there is a neighbourhood $\V$ of $F$ in $\Diff^4 (M)$ such that $W_j$ is correctly aligned with $W_{j+1}$ under $\widetilde{F}^{n_j}$ for each $j=1, \ldots, N-1$ and each $\widetilde{F} \in \V$. Moreover, our assumptions on $F$ imply that there is a $K>0$ such that the neighbourhood $\V$ is of order $\epsilon^K$, in the sense that there is some $R>0$ independent of $\epsilon$ such that the ball of radius $R \epsilon^K$ centred at $F$ in $\Diff^4 (M)$ is contained in $\V$. Therefore there is some $a^*>0$ (that may depend on $\epsilon$) such that if we define
\begin{equation}
\Xi^* = \left[ \xi^*_{1,1} - a^*, \xi^*_{1,1} + a^* \right] \times \cdots \times \left[ \xi^*_{1,\ell_2} - a^*, \xi^*_{1,\ell_2} + a^* \right] \subset \mathrm{Int} \left( [0,1]^{\ell_2} \right),
\end{equation}
where $\xi_1^* = \left( \xi_{1,1}^*, \ldots, \xi_{1,\ell_2,}^* \right)$, then $\widetilde{G} ( \cdot ; \xi) \in \V$ for all $\xi \in \Xi^*$. Moreover, if we choose $L \in \mathbb{N}$ large enough, we may assume that $G( \cdot, \theta, \xi) = \widetilde{G} ( \cdot ; \xi) + O(\epsilon^L) \in \V$ for all $\theta \in \mathbb{R}^{\ell_1}, \xi \in \Xi^*$. 

For each $j=2, \ldots, N$ let
\begin{equation}
\Omega_j = \sum_{k=1}^{j-1} n_k.
\end{equation}
Since $L \in \mathbb{N}$ is sufficiently large, we can find positive constants $C_j$ (independent of $\epsilon$) such that, if we let $\Xi_1 = \{ \xi_1^* \}$, and
\begin{equation}
\Xi_j = \left[ \xi_{1,1}^* - C_j \Omega_j \epsilon^L, \xi_{1,1}^* + C_j \Omega_j \epsilon^L \right] \times \cdots \times \left[ \xi_{1,\ell_2}^* - C_j \Omega_j \epsilon^L, \xi_{1,\ell_2}^* + C_j \Omega_j \epsilon^L \right]
\end{equation}
for $j=2, \ldots, N$, then
\begin{equation}
\Xi_j \subset \Xi^*
\end{equation}
for each $j=1, \ldots, N$, and the $\xi$ component of $\Psi^{n_j}(z, \theta, \xi)$ lies in $\mathrm{Int} \left( \Xi_{j+1} \right)$ for each $z \in W_j, \theta \in \mathbb{R}^{\ell_1}, \xi \in \Xi_j$. 

Choose any $\zeta_1^{\pm} \in \mathbb{R}$ such that $\zeta_1^- < \zeta_1^+$, and consider the rectangle
\begin{equation}
\Theta_1 = \left[ \zeta_1^-, \zeta_1^+ \right]^{\ell_1} \subset \mathbb{R}^{\ell_1}. 
\end{equation}
For each $j=2, \ldots, N$ we can find $\zeta_j^{\pm} \in \mathbb{R}$ such that $\zeta_j^- < \zeta_j^+$, and such that if 
\begin{equation}
\Theta_j = \left[ \zeta_j^-, \zeta_j^+ \right]^{\ell_1},
\end{equation}
then the $\theta$ component of $\Psi^{n_j}(z, \theta, \xi)$ lies in $\mathrm{Int} (\Theta_{j+1})$ for each $z \in W_j, \theta \in \Theta_j, \xi \in [0,1]^{\ell_2}$.

Now, for each $j=1, \ldots, N$ choose the entry and exit sets of the rectangles $\Theta_j, \Xi_j$ to be
\begin{equation}
\Theta_j^+ = \partial \Theta_j, \quad \Theta_j^- = \emptyset, \quad \Xi_j^+ = \partial \Xi_j, \quad \Xi_j^- = \emptyset.
\end{equation}
Define 
\begin{equation}
\W_j = W_j \times \Theta_j \times \Xi_j \subset \widehat{M}
\end{equation}
where the entry and exit sets $\W_j^{\pm}$ are given by the product formula \eqref{eq_productentryexit}. Clearly the sets $\W_j$ are windows. By construction, $\W_j$ is correctly aligned with $\W_{j+1}$ under $\Psi^{n_j}$ for each $j=1, \ldots, N$ since the error terms are small (by assumption [B1]), and
\begin{equation} \label{eq_hatwindowsneartori}
d \left( (z, \theta, \xi), \widetilde{\L}_j \right) < \eta
\end{equation}
for each $(z, \theta, \xi) \in \W_j$ if $L$ is sufficiently large, where $\widetilde{\L}_j = \widetilde{\L}(p^*_j, \xi^*_j)$ with $p^*_1, \ldots, p^*_N$ and $\xi^*_1$ as chosen earlier, and for some $\xi^*_2, \ldots, \xi^*_N \in [0,1]^{\ell_2}$. Therefore, combining \eqref{eq_hatwindowsneartori} and Theorem \ref{theorem_onecansee}, we see that there are $w_1, \ldots, w_N \in \widehat{M}$ such that
\begin{equation}
w_{j+1} = \Psi^{n_j} (w_j)
\end{equation}
and
\begin{equation}
d \left(w_j, \widetilde{\L}_j \right) < \eta. 
\end{equation}
Moreover the time estimate \eqref{eq_timeestimate2} follows from the time estimate \eqref{eq_timeestimate1}.

\bibliographystyle{abbrv}
\bibliography{caw_diffusion_refs} 

\end{document}